\documentclass[11pt]{amsart}

\usepackage{graphicx}
\usepackage{xcolor} 
\usepackage{amsmath,amssymb,amsfonts,euscript,enumerate}
\usepackage{amsthm}  
\usepackage{color,wrapfig}
\usepackage{pxfonts}
\usepackage{verbatim}
\usepackage{appendix}

\newcommand{\N}{\mathbb N}
\newcommand{\R}{{\mathbb R}}

\newcommand{\e}{\varepsilon}

\newcommand{\D}{\nabla}
\newcommand{\Di}{\nabla_{i}}

\newcommand{\fr}{\frac}
\newcommand{\p}{\partial}
\newcommand{\rhoe}{\rho^\varepsilon}

\newcommand{\Je}{J^\varepsilon}
\newcommand{\Jes}{J^{\varepsilon 2}}
\newcommand{\ue}{u^\varepsilon}

\newtheorem{thm}{Theorem}[section]
\newtheorem{lem}[thm]{Lemma}
\newtheorem{prop}[thm]{Proposition}
\newtheorem{definition}{Definition}[section]

\newtheorem{remark}[thm]{Remark}
\newtheorem{corollary}{Corollary}[section]
\numberwithin{equation}{section}

\newcommand{\no}[1]{}

\begin{document}

\title{Multi-D Fast Diffusion Equation via Diffusive Scaling of Generalized Carleman Kinetic Equation}

\author{Beomjun Choi}
\address{Beomjun Choi : 
Department of Computational Mathematics and Imaging, National Institute for Mathematical Sciences, Daejeon 305-811, Korea \&
Department of Mathematics, Columbia University, 2990 Broadway, New York, NY 10027, USA}
\email{bc2491@columbia.edu}

\author{Ki-Ahm Lee}
\address{Ki-Ahm Lee :
Department of Mathematical Sciences, Seoul National University, Seoul 151-747, Korea \&
Department of Mathematics, Korea Institute for Advanced Study, Seoul 130-722, Korea}
\email{kiahm@ snu.ac.kr}

\date{}
\begin{abstract}
In this paper, we investigate generalized Carleman kinetic equation for n$\ge$2 and prove convergence towards the solution of equation with fast diffusion or porous medium type, $u_t=\Delta u^m$ ($0\le m\le2$), in its diffusive hydrodynamic limit. Using comparison principle of system combined with fixed speed propagation property of transport equation, we create a new barrier argument for this hyperbolic system. It is crucial to construct explicit local sub and solution of system and this is done by employing an ansatz from second order asymptotic expansion. This allow us prove diffusive limit toward subcritical FDE, which is thought to be difficult with previous method due to the lack of mass conservation. Moreover, we can also prove convergence with growing initial data in slow diffusion range (including $n=1$), which was also unknown before.

\end{abstract}

\maketitle

\tableofcontents
\newpage

\section{Introduction}

In this paper, we are going to consider discrete velocity Boltzmann type equation of a fictitious gas proposed by Carleman \cite{C,C2}.
To describe Carleman's model equation in higher dimension,  let us choose usual coordinate $x=(x_1,\cdots,x_n)\in \R^n$ and  the standard orthonormal basis $\{v_i:\, i=1, \cdots,n\}$  in $\R^n$. And $v_{i+n}$ is the opposite direction of $v_i$, which means $v_{i+n}=-v_i$. Now $u_i$ represents density of particles moving in $v_i$ velocity for $i=1, \cdots, 2n$.

 The multidimensional Carleman  kinetic system can be written as the following: for $n\geq 2$,
  \begin{equation}\begin{aligned}\label{eq-main}
&\text{ $\begin{cases}
{\p}_{t}u^{\e}_{i}+\fr{1}{\e}\D_{{i}}u_{i}^\e=\fr{1}{2n\e^2}\sum_{j=1}^{2n} k(u_{j}^\e,u_{i}^\e,x)(u_{j}^\e-u_{i}^\e)\quad & \text{in $Q$}\quad (i=1,\cdots,n)  \\
{\p}_{t}u^{\e}_{i+n}-\fr{1}{\e}\D_{{i}}u_{i+n}^\e=\fr{1}{2n\e^2}\sum_{j=1}^{2n} k(u_{j}^\e,u_{i+n}^\e,x)(u_{j}^\e-u_{i+n}^\e)\quad & \text{in $Q$} \quad (i=1,\cdots,n)
\end{cases}$}
\\& \text{with initial condition }u^\e_i=g_{i}\ge 0\quad\text{at $t=0$} \quad (i=1,\cdots,2n).
\end{aligned}
\end{equation}In a simpler form of kinetic equation, this can be rewritten  as follows:
\begin{equation}\label{eq-main0}
\begin{aligned}
&\text{${\p}_{t}u^{\e}_{i}+\fr{1}{\e}v_i\cdot \D u_{i}^\e=\fr{1}{2n\e^2}\sum_{j=1}^{2n} k(u_{j}^\e,u_{i}^\e,x)(u_{j}^\e-u_{i}^\e)\quad\text{in $Q$}\quad (i=1,\cdots,2n)$.}
\end{aligned}
\end{equation} 

Physically, $\e$ represents the order of Knudsen number and Mach number in the Boltzmann type equation. Here, Knudsen number is a dimensionless ratio of mean free path to a representative physical length scale while Mach number is the dimensionless ratio of bulk velocity to the reference speed. Roughly, the equation \eqref{eq-main} considers the case when particles collide with other particles after it travels with average distance $\e$, and their bulk velocity is of order $\e$ with respect to. the reference speed.

The limit $\e\to0^+$ is called hydrodynamic or diffusive limit. In this limit, this mesoscopic kinetic equation for rarefied gas particles becomes a certain continuum equation.
In another point of view, the equation \eqref{eq-main} can be derived from the standard case $\e=1$ by taking parabolic scaling $(x,t)\to(\e x, \e^2 t)$ capturing the diffusive behavior in  the equation, and this explains why this limit is called diffusive limit \cite{SR}.

In the equation, $k(u_{j}^\e,u_{i}^\e,x)$ describes the interaction rate or collision rate between  two group of particles, $i$ and $j$ at a position $x$.
 Since $u_i$ is the density of particles with velocity $v_i$,
interaction between $i$ and $j$ results  in velocity changes among particles in these two groups and  it is modeled as a reaction term on the right hand side of \eqref{eq-main}.     Even if the reaction rate k(a,b,x) can be as general as possible, our main interest is the case $k(a,b,x)=k_\alpha(a,b)=(a+b)^\alpha$, $|\alpha| \le 1$. We call  $k_{\alpha}(a,b)$ as a generalized Carleman type reaction rate since the case for $\alpha=1$ and $ n=1$  was  investigated by Carleman for the first time.

After Carleman's model was proposed in \cite{C,C2}, general $\alpha$ model was introduced at \cite{PT}. It has been extended to 3 dimensional space at \cite{TL}, which is sightly different from \eqref{eq-main}. At \eqref{eq-main}, we consider the interaction between each pair of $u_i$s while \cite{TL}  includes interaction only between $u_i$ and the average of $u_i$s. The main mathematical advantage of our model \eqref{eq-main} is that a kind of comparison principle  between two vector valued solutions holds
 for larger class of interaction rates, $k$.
 We will study the rigorous convergence of $u_i$ to its hydrodynamic limit, $\frac{1}{2n}\rho$ as $\e\rightarrow 0$ for all $i=1,\cdots, 2n$ and the limit equation satisfied by $\rho$:

    \begin{equation}\label{eq-main-limit}
    \rho_t=\D\cdot \left(\fr{1}{n k(\fr{\rho}{2n},\fr{\rho}{2n},x)} \D \rho\right).
    \end{equation}

  Finally, it is believed that our model could be used in various modeling problem whose scale is located between kinetics and continuum theory. This model and method have enough flexibility to adopt various situations. For instance, one create model with more than one kinds of molecules or species by increasing velocity components and describing reaction between different molecules separately. We expect our key technique and strategy hold in a similar fashion. Moreover, it could have infinite(continuous) velocity components. See \cite{TL}, for instance, for possible extensions towards this direction. One suitable application is in Mathematical Biology such as chemotaxis or evolution/population dynamics. In chemotaxis, cell dynamics is not in the range of continuum equation,
  which can be models by a kinetic equation and then   standard Keller-Segel Model can be derived by taking drift-diffusion limit. Currently we try to prove certain degenrate nolinear Keller-Segel can be approximated by a  kinetic chemotaxis model . Please refer \cite{CAB} and \cite{HKS} for previous works regarding this.

\subsection{History} As mentioned earlier, the first model was first introduced at \cite{C,C2}. \cite{K}, \cite{McK} considered hydrodynamic limit of Carlemann kinetic equation($\alpha=1$, $n=1$) in the semigroup or probabilistic framework. In \cite{K}, he assumed some zero flux condition on function space to prove semigroup convergence of Carleman equation toward logarithmic diffusion equation. Later in \cite{TL} and \cite{PT}, they studied hydrodynamic limit of this equation for $n=1$, $\alpha<1$ with some weighted $L^1$ integrability condition on initial data. Also in \cite{TL}, they considered 3-D generalization which is slightly different model of ours (at \cite{TL}, each density interacts with mean density of others).
The most notable result recently is in \cite{SV}, where authors incorporate $n=1$ and $|\alpha|\le1$ including $\alpha=1$, and remove other conditions on initial data except for plain $L^1$ integrability. We regard our work as a direct generalization of \cite{SV}.

In $n=1$ equation, \cite{ST} deals with $1<\alpha <\fr{4}{3}$ where its target diffusion equation corresponds to a ultrafast diffusion equation. Reaction $k_\alpha$ with $\alpha>1$, has no $L^1$ contraction property, which could be one of main  difficulties. Fortunately, however, they could get some global entropy estimate by assuming stronger integrability condition on initial data as \cite{TL} did. To get global entropy estimates in \cite{TL},\cite{ST}, those weighted $L^1$ norms should be controlled for a given time interval $\e$ independently. And this implies the target diffusion equation has mass conservation property. The fortune comes for the fact that those fast and ultrafast diffusion equations in $n=1$ have unique maximal solution which has mass conservation property. Without this mass conservation property of target equation, it is not clear wether  we can get similar global entropy estimates done in \cite{TL}\cite{ST}, which makes a main analytic difficulty in our equation. In $n\ge2$, $\alpha\ge \fr{2}{n}$, \eqref{eq-target} no longer has mass conservation and the solutions with $L^1$ initial data will extinct in finite time.

With newly developed technique and Barenblatt type lower bound assumption on the initial data, which looks one natural assumption in FDE literature, we could prove uniform local entropy estimate and this gives local estimates on fluxes. Then we follow compensated convergence argument to prove their diffusive limit are fast diffusion equations.

In different perspectives, there are several results on initial-boundary value problems of \eqref{eq-main}. \cite{F2} proved existence and diffusive limit of the solution with reflecting boundary condition using semigroup theory. Regarding the problem of Dirichlet type, \cite{F} showed the existence of a kind of weak solution and \cite{GS} proved the diffusive limit of the solutions by controlling a relative entropy. On the other hand, currently we are working on a new initial-boundary value problem. On higher dimension setting, our problem assumes prescribed the directional fluxes data. This is generalization of \cite{F2} since it corresponds to $n=1$ and zero flux case. This new condition is seemingly physical assumption and linked to parabolic problem of Neumann type in the limit.

\subsection{Main Theorems}
First, let us  define variables $\rho^\e_{i}=u^\e_i + u^\e_{i+n}, \ \Je_{i}=\fr{1}{\e}(u^\e_i-u^\e_{i+n}),\ \Je=(\Je_{1},\cdots,\Je_{n}),\ \rhoe=\sum_{i=1}^{n}\rhoe_{i}$, and similarily $\rho^\e_{i,j}=u^\e_i + u^\e_{j}, \ \Je_{i,j}=\fr{1}{\e}(u^\e_i-u^\e_{j})$ for future usage. We can rewrite \eqref{eq-main}

\begin{flalign}\label{eq-main2}
\begin{cases}
\p_t \rhoe_i + \D_{i}\Je_i = \fr{1}{2n\e^2}\sum_{j=1}^{2n} \left(k(u_j^\e,u_i^\e,x)(u_j^\e-u_i^\e)+k(u_j^\e,u_{i+n}^\e,x)(u_j^\e-u_{i+n}^\e) \right)\\
\e^2\p_t \Je_{i}+\D_{i}\rhoe_{i}=\fr{1}{2n\e}\sum_{j=1}^{2n}\left((k(u_j^\e,u_i^\e,x)-k(u_j^\e,u_{i+n}^\e,x))u_j^\e\right.
\\\quad\quad\quad\quad\quad\quad\quad\quad\quad\quad\left. -(k(u_j^\e,u_i^\e,x)u_i^\e-k(u_j^\e,u_{i+n}^\e,x)u_{i+n}^\e)\right),
\end{cases}
\end{flalign}
and this implies
\begin{flalign}\label{eq-main3}&
\begin{cases}
\p_t\rhoe + div \Je=0 \\
\e^2\p_t\Je_i+\fr{1}{n}\D_i\rhoe=\fr{1}{2n} \sum_{j=1}^{2n}\left[(k(u_j^\e,u_i^\e,x)-k(u_j^\e,u_{i+n}^\e,x))\left(\fr{2u_j^\e-u_i^\e-u_{i+n}^\e}{2\e}\right)\right.
\\\qquad\left.-(k(u_j^\e,u_i^\e,x)+k(u_j^\e,u_{i+n}^\e,x))\left(\fr{u_i^\e-u_{i+n}^\e}{2\e}\right) \right ]+\fr{1}{n}\D_i\left(\sum_{j=1}^{n}(\rhoe_j-\rhoe_i)\right ).
\end{cases}&
\end{flalign}

Under a suitable condition  on interaction rate $k$ and  initial data (See Definition \ref{def-adm-interaction} and Theorem \ref{thm-dl-1}),
we prove $\Je_{i,j} \rightharpoonup J_{i,j}, u^\e_{i} \to \fr{\rho}{2n}$ for some $J_{i,j}$s and $\rho$ in $L^{2}_{loc}$ in (x,t) as $\e \to 0$ for sufficiently general class of $k$s. From \eqref{eq-main3}, this immediatly implies
\begin{equation}\begin{aligned}
&\p_{t}\rho+divJ=0 \quad \text{ and}\\
&\fr{1}{n}\D_{i}\rho = -k(\fr{\rho}{2n},\fr{\rho}{2n},x)J_i &(i=1,\cdots,n)
\end{aligned}\end{equation}
in distribution sense.
So, we can prove $\rho$ is a weak solution of $\p_t\rho - div(\fr{1}{n\cdot k(\fr{\rho}{2n},\fr{\rho}{2n},x)}\D\rho)=0 $ in $\R^n\times(0,\infty)$. Moreover, this $\rho$ has initial value $\rho(0)=\sum_{i=1}^{2n}g_i$ in trace sense.

In our main cases $k(a,b)=k_{\alpha}(a,b)=(a+b)^\alpha$, we assume more general condition on initial data, Definition \ref{def-adm}, than the condition assumed in Theorem \ref{thm-dl-1}. 
Then limit $\rho$ also solve corresponding Fast Diffusion Equation or Porous Medium Equation with corresponding initial data:
\begin{equation}\label{eq-target}
 \partial_t\rho-\D\cdot\left(\fr{1}{n^{1-\alpha} \cdot \rho^\alpha}\D \rho\right) =0 \text{ in $\mathcal{D}$ with } \rho(0)=\sum_{i=1}^{2n}g_i.
\end{equation}

To state this condition on initial data precisely, let us define a family of function spaces $X_{n,\alpha}$, a collection of $L^1$-perturbations of functions satisfying certain decay and growth estimates at the infinity.
\begin{definition}\label{def-adm}
$X_{n,\alpha} $ is  a collection of   admissible nonnegative initial data $ g \in L^1_{loc}(\R^n)$  such that there exists nonnegative $f \in L^1_{loc}(\R^n)$ with $f-g\in L^1(\R^n)$ and f satisfies 
\begin{equation}\begin{aligned}
&\liminf_{|x|\to\infty}|x|^\fr{2}{\alpha}f >0 & \text{if   } \fr{2}{n}\le \alpha \le 1 \\
\end{aligned}
\end{equation}

and
\begin{equation}\begin{aligned}
&{\limsup_{|x|\to\infty}|x|^{-2}\log (f+1) <\infty} && \text{if   } 0\le\alpha \le1 \\
& \limsup_{|x|\to\infty}|x|^\fr{2}{\alpha}f <\infty && \text{if   } -1\le\alpha<0
\end{aligned}
\end{equation}
\end{definition}

Our main result could be summarized in the following theorem.
\begin{thm}\label{thm-mmainthm} For the interaction rates $k(a,b,x)=(a+b)^\alpha$ with  $|\alpha|\le1$,
suppose $\{u_i^\e\}$ are unique mild solutions of \eqref{eq-main0} (see definition \ref{def-mild}) with initial data $g_i\in X_{n,\alpha}$.
\begin{enumerate}[(i)]
\item

Then, there exists $\overline{ T}= \overline{T}(n,\alpha, \{g_i\})\in (0,\infty]$ such that
\begin{equation}
u^\e_i \to \fr{\rho}{2n}\text{ in }L^1_{loc}(Q_{{T}})\text{ for all }0<T<\overline{T}\text{ as }\e\to0.
\end{equation}
If solution of \eqref{eq-fastdiff} have no uniqueness for the initial data $\sum_{i=1}^{2n} g_i$, the convergence takes along a subsequence  for each given sequence $\e_j\to0$. Otherwise, the convergence is arbitrary as $\e\to0$.

\item Moreover, if initial data $g_i$ has decomposition  \begin{equation}\label{eq-compactcondition}g_i=l_i + h_i,\text{ }h_i\in L^1(\R^n),\text{ and }\ l_i \in C^1(\R^n) \text{ with }\int |Dl_i| < \infty  \end{equation}, then \begin{equation}\rho^\e \to \rho\text{ in }C([0,T], L^1_{loc}(\R^n))\text{ for all } 0<T<\overline{T}.\end{equation}

\item When $\alpha\neq1$, $\rho\in C([0,T], L^1_{loc}(\R^n))$ is the unique weak solution of $\partial_t\rho-\D\cdot(\fr{1}{n^{1-\alpha} \cdot \rho^\alpha}\D \rho) =0$ with initial data $\sum_{i=1}^{2n} g_i$. In case $\alpha=1$, $\rho\in C([0,T], L^1_{loc}(\R^n))$ is some weak solution of $\partial_t\rho-\D\cdot(\fr{1}{\cdot \rho}\D \rho) =0$ satisfying $\liminf_{|x|\to\infty}|x|^2\rho >0$ for all $t\in[0,T]$ with initial data $\sum_{i=1}^{2n} g_i$.

\item $\overline{T}$ could be taken as $\overline{T}= C(n,\alpha)\cdot min(T_1,T_2)$, $T_i\in (0,\infty]$ where 

$$T_1=
\begin{cases}
\begin{aligned}
&[min_{i=1}^{2n}( \liminf_{|x|\to\infty}|x|^\fr{2}{\alpha}f_i )]^\fr{1}{\alpha} && \text{if   }\ \fr{2}{n}\le\alpha\le 1\\
&\qquad\qquad\infty && \text{if   }-1\le\alpha<\fr{2}{n}
\end{aligned}
\end{cases}\quad$$and

$$T_2=
\begin{cases}
\begin{aligned}
&[max_{i=1}^{2n}\limsup_{|x|\to\infty}|x|^{-2}\log (f_i+1)]^{-1}&& \text{if   }\ 0\le\alpha\le 1\\
&[max_{i=1}^{2n}( \limsup_{|x|\to\infty}|x|^\fr{2}{\alpha}f_i)]^\fr{1}{\alpha} && \text{if   } -1\le\alpha<0.
\end{aligned}
\end{cases}$$ Here, negative power of 0 is defined to be $\infty$ and $f_i$ is a function with
$f_i-g_i\in L^1$ given from definition \ref{def-adm}.
 
\end{enumerate}
\end{thm}

\subsection{Outlines}
First in the next two sections, we are going to prove some basic facts of the first oder kinetic equation, \eqref{eq-main}. Definition and existence of weak solution of \eqref{eq-main} will be proved via some apriori estimates and several versions of comparison priciple which we use in most situations throughout this paper, will be introduced.

Next in Section 4, we prove so called local entropy estimate, $\e$-independent local flux $J^\e$ $L^2$ estimate, which allow us to pass the limit in section 5. This entropy estimate is proved in a restricted case when we known our solution of \eqref{eq-main} is locally bounded above and below with positive numbers $\e$ independently. Diffusive limit toward target diffusion equation in these cases will be presented. The estimate and way to pass the limit is just multi-dimensional version of the technique done in \cite{SV}. Upto section 5, we deal with general admissible reaction rates $k(a,b,x)$.

In Section 6 and 7, we turn to interaction rates $k_\alpha(a,b)=(a+b)^\alpha$, which give us FDE or PDE in the limit. Main goal is to broaden the function class of initial data which we can prove diffusive limit convergence. Upto section 5, the only case we know a priori positive local upper and lower bounds is when our initial data are globally bounded above and blow with positive constants. Instead, we assume some specific profile of initial data and prove that solutions of \eqref{eq-main} keep this shape of profile locally in space-time uniformly in $\e$. We show this by employing an ansatz from the second order asymptotic expansion method and delicate use of comparison principle which control the influence from infinite point. Section 6 is devoted to the heuristic computation of our ansatz and showing comparison technique in exemplary case. The complete proof of Theorem \ref{thm-mmainthm} is given in Section 7.

\subsection{Notations and Definitions}

First, if there is no specific mention, the class of interaction rate $k$s in this paper is as follows.
\begin{definition}\label{def-adm-interaction}
Interaction rate k(u,v,x) is admissible if 
\begin{enumerate}
\item k(u,v,x) is measurable real function on $\R_+\times\R_+\times\R$ such that it is symmetric in u,v arguments i.e. $k(u,v,x) =k(v,u,x)$
\item for every $\lambda>0$ there exist $M=M(\lambda)>0$ such that $\fr{1}{M} \le k(u,v,x) \le M$ for all $\fr{1}{\lambda}\le u,v \le \lambda$, and $x\in\R$.
\item k is continuous as a function of u and v for a.e. x. Moreover, function $k(u,v,x)(u-v)$ is uniformly Lipschitz continuous in  $(u,v)$ on $K\times K$ for all $K\subset\subset \R^n_+$.
\end{enumerate}
\end{definition}

These will be some notations that appear frequently.
\begin{enumerate}
\item $Q = \R^n \times (0,\infty)$, $Q_T = \R^n \times (0,T)$, $T>0$, $Q_{R,T}=B_{R}(0)\times (0,T))$
\item $k_\alpha(a,b)= (a+b)^\alpha,\ \text{for}\  a,b>0,\ \alpha\in\R$
\item $sgn^+(a)= \begin{cases}1\quad a>0 \\ 0 \quad a\le0\end{cases}$
\item $v_i=\begin{cases}\overrightarrow{e}_i\quad 1\le i \le n\\ -\overrightarrow{e}_{i-n}\quad n+1\le i \le 2n\end{cases}$
\item $\overrightarrow{\nu}(x)=(\nu_1,\nu_2,\ldots,\nu_n)\in \R^n $ is a unit normal vector to $\partial\Omega$ at $ x\in \partial \Omega$
\item $\rho^\e = \sum_{i=1}^{2n} u^\e_i$, $\rho^\e_i=u^\e_i+u^\e_{i+n}$, $\rho^\e_{i,j}=u^\e_i+u^\e_j$
\item $J^\e=(J^\e_1,J^\e_2,\ldots,J^\e_n)$, $J^\e_i= \fr{\ue_i-\ue_{i+n}}{\e}$, $J^\e_{i,j}=\fr{\ue_i-\ue_{j}}{\e}$

\end{enumerate}
\section{Existence Theory}
Existence of global weak solution comes from a simple  modification of Section 2 of \cite{TL}. We also have an existence of weak solution  with locally bounded initial data, which is modification of Salvarani and Vazquez \cite{SV}.

\begin{lem}\label{lem-Apriori} (A Priori Estimate of (1.1))
If $ u_i \in L ^ {\infty} (\R ^n\times [0,T]) $ is compactly supported solution of (1.1), then, for a convex function $ \phi $ on $ \R $ with  $\phi(0)=0 $,
$$ \int_{\R^n} \sum_{i=1}^{2n} \phi(u_i)\ dx\ \ \text{is nonincreasing in $t\in [0,T]$ .}$$

\begin{proof}
Let us multifly  $\phi'(u_i)$ on \eqref{eq-main}. Taking a summation in $i$  and integration w.r.t. $x,t$, we get
\begin{equation}
\int_{\R^n} \sum_{i=1}^{2n} \phi(u_i)\ dx\bigg\arrowvert^{t_2}_{t_1} + 0 = \fr{1}{2n\e^2}\sum_{i,j=1}^{2n} k(u_{j}^\e,u_{i}^\e,x)(u_{j}^\e-u_{i}^\e)(\phi'(u_i)-\phi'(u_j))\le 0.
\end{equation}
\end{proof}
\end{lem}

Moreover, following argument of \cite{SV}, we can notice that the system (1.1) propagates along characteristics with speed $ \fr{1}{\e}$. Therefore any weak solution on a cube $\prod_{i=1}^{n} [a_i,b_i] $ at some time $t>0$ is completely determined by initial data on $\prod _{i=1}^{n} [a_i-\fr{t}{\e},b_i+\fr{t}{\e}]$. To prove a priori the $L^{\infty}$ bounds of any solutions with $L^{\infty}$ initial data, we can  approximate $L^{\infty}$ initial data by functions with expanding compactly supports and then apply  Lemma \ref{lem-Apriori}  for  $\phi(x)=|x|^p$. Now by taking  $p \to \infty$ on compact domains and expanding domains,  we will get a priori $L^{\infty}$ bounds of any solutions with $L^{\infty}$ initial data. The same $L^{\infty}$-estimate comes from a  comparison with a constant solution in the next section. This $L^\infty$ a priori estimate allows us to prove the global existence of a unique weak solution with $L^\infty$ initial data. (c.f. Theorem 1 of \cite{SV} thm1, Proposition 2.4 of \cite{TL} ).

\begin{thm}
(Global Existence of Weak solution with $L^{\infty}\cap L^1$initial data)(\it{c.f}. \cite{SV} Thm1, \cite{TL} Prop 2.4)
\item If $g_i\in L^\infty\cap L^1(\R^n)$ for each $i$ and $g_i\geq 0$, then  (1.1) has a global unique weak solution $u_i \in L^\infty(Q) \cap C([0,\infty), L^p (\R^n))$ for all $p\ge1$. By saying weak solution, it means a solution in distribution sense and its initial data is recovered in trace sence as $t \to 0+$.
\end{thm}
This $L^{\infty}\cap L^1$ existence theorem and finite speed of propagation also imply the following existence theorem with $L^{\infty}_{loc}$ initial data.

\begin{definition}
A weak solution of \eqref{eq-main} with initial $0\le g_i \in L^\infty_{loc}(\R^n)$ is a family of functions $(u_i)_{i=1}^{2n}\in C([0,T],L^\infty_{loc}(\R^n))\cap L^1_{loc}(Q_T)$, $T>0$, such that \eqref{eq-main} is satisfied in the sense of distribution and its initial data is recovered in the trace sense.
\end{definition}

\begin{thm}
(Global Existence of Weak solution with $L^{\infty}_{loc}(\R^n)$ initial data) (\it{c.f}. \cite{SV} Prop 2) If $0\le g_i \in L^{\infty}_{loc}(\R^n)$, system \eqref{eq-main} has a global unique weak solution $(u_i)_{i=1}^{2n}\in C([0,\infty),L^p_{loc}(\R^n))$ for all $p\ge1$.
\end{thm}

Finally, we want to show  that this vector valued solution preserve their positivity and a positive uniform lower bound.
\begin{lem}\label{lem-Apriori2}
If $0\le N\le g_i \in L^{\infty}_{loc}(\R^n)$, then $N\le u_i $ for all $i=1,\ldots,2n$.
\begin{proof}

Choose smooth enough convex function $\phi_N$ such as $$\phi(x)=\begin{cases}|x-N|^4&\quad \text{if }x\le N\\ 0 &\quad \text{otherwise}\end{cases}$$
and rest steps are similar to Lemma \ref{lem-Apriori}.

For compactly supported bounded solution ({\it i.e.} when $g_i$ satisfying this condition), $$\int_{\R^n} \sum_{i=1}^{2n} \phi_N(u_i)\ dx$$ is nonincreasing in $t$, but this is zero at $t=0$. This implies $N\le u_i$. General $L^\infty_{loc}$ case follows from previous standard argument using finite speed of propagation.
\end{proof}
\end{lem}
\section{Comparison Principle}
\begin{definition}\label{def-dissipativity}(T-dissipaticity of interaction term)({\it c.f.} \cite{SV})
\item
Interaction term, RHS of\eqref{eq-main}, is called T-dissipative if $$A(x):\R^{2n}_+ \longrightarrow \R^{2n},\text{ defined as}$$
\begin{equation}A(x)\left(\begin{array}{c}u_1\\u_2\\\vdots\\u_{2n}\end{array}\right)=
\left(\begin{array}{c}\sum_{j=1}^{2n}k(u_j,u_1,x)(u_j-u_1)\\ \sum_{j=1}^{2n}k(u_j,u_2,x)(u_j-u_2)\\ \vdots \\\sum_{j=1}^{2n}k(u_j,u_{2n},x)(u_j-u_{2n}) \end{array}\right)
\end{equation}
satisfies  $(A\overrightarrow{u}-A\overrightarrow{v})\cdot sgn^+(\overrightarrow{u}-\overrightarrow{v})\le 0$\quad for all $\overrightarrow{u}$ and $\overrightarrow{v}$ at $x$ a.e.
\end{definition}
\begin{lem}\label{Lem-Tdissipativity}
(T-dissipativity of $k_\alpha$\text{ for }$|\alpha|\le 1$)
\item
If $|\alpha|\le 1$, $k_\alpha(a,b)=(a+b)^\alpha$ is T-dissipative.
\end{lem}
\begin{proof} Let us define
$$A_\alpha:\R^{2n}_+ \longrightarrow \R^{2n}$$
as
\begin{equation}A_\alpha\left(\begin{array}{c}u_1\\u_2\\\vdots\\u_{2n}\end{array}\right)=
\left(\begin{array}{c}\sum_{j=1}^{2n}(u_j+u_1)^\alpha(u_j-u_1)\\ \sum_{j=1}^{2n}(u_j+u_2)^\alpha(u_j-u_2)\\ \vdots \\\sum_{j=1}^{2n}(u_j+u_{2n})^\alpha(u_j-u_{2n}) \end{array}\right).
\end{equation}
Then,
\begin{equation}\begin{aligned}(A_\alpha\overrightarrow{u}-A_\alpha\overrightarrow{v})\cdot sgn^+(\overrightarrow{u}-\overrightarrow{v})\\=\sum_{i,j}[(u_j+u_i)^\alpha(u_j-u_i)+(v_j+v_i)^\alpha(v_i-v_j)]sgn^+(u_i-v_i)\\
\triangleq\sum_{i,j}K_{u,v}(i,j)=\fr{1}{2}\sum_{i,j}K_{u,v}(i,j)+K_{u,v}(j,i)
\end{aligned}
\end{equation}

Set $f(x,y):=(x+y)^\alpha(x-y)$ for $x,y \in \R$. Notice that $\partial_x f \ge 0$ and $\partial_y f \le 0$ on $x,y>0$ if $|\alpha|\le1$. Thus, in case $u_j\le v_j$ and $u_i \ge v_i$, $K_{u,v}(i,j)=f(v_j,v_i)-f(u_j,u_i)\le0 $.

Therefore,  for each fixed $i$ and $j$, we have
\begin{displaymath}
\begin{cases}\text{if }u_i\ge v_i\ \&\ u_j\ge v_j\text{, \quad}K_{u,v}(i,j)=-K_{u,v}(j,i) \Rightarrow K_{u,v}(i,j)+K_{u,v}(j,i)=0\\
\text{if }u_i\le v_i\ \&\ u_j\le v_j\text{, \quad}K_{u,v}(i,j)=K_{u,v}(j,i)=0 \Rightarrow K_{u,v}(i,j)+K_{u,v}(j,i)=0\\
\text{if }u_i\ge v_i\ \&\ u_j\le v_j\text{, \quad}K_{u,v}(j,i)=0\,\& \, K_{u,v}(i,j)\le 0\Rightarrow K_{u,v}(i,j)+K_{u,v}(j,i)\le 0\\
\text{if }u_i\le v_i\ \&\ u_j\ge v_j\text{, \quad}K_{u,v}(i,j)=0\,\& \, K_{u,v}(j,i)\le 0 \Rightarrow K_{u,v}(i,j)+K_{u,v}(j,i)\le 0
\end{cases}
\end{displaymath}

Now we have $(A_\alpha\overrightarrow{u}-A_\alpha\overrightarrow{v})\cdot sgn^+(\overrightarrow{u}-\overrightarrow{v})\le 0$.
\end{proof}
In order to state comparison principle in general form, it is useful to
introduce notions of subsolution and supersolution.

\begin{definition}\label{def-sol}
$(u_i)_{i=1}^{2n}\in C([0,T],L^\infty_{loc}(\R^n))\cap L^1_{loc}(Q_T)$ is a weak subsolution (or supersolution)  of \eqref{eq-main} if
\begin{equation}
\begin{cases}
{\p}_{t}u^{\e}_{i}+\fr{1}{\e}\D_{{i}}u_{i}^\e-\fr{1}{2n\e^2}\sum_{j=1}^{2n} k(u_{j}^\e,u_{i}^\e,x)(u_{j}^\e-u_{i}^\e)\le 0\ (\ge 0) \text{ $ i=1,2,\ldots,n$}\\
{\p}_{t}u^{\e}_{i+n}-\fr{1}{\e}\D_{{i}}u_{i+n}^\e-\fr{1}{2n\e^2}\sum_{j=1}^{2n} k(u_{j}^\e,u_{i+n}^\e,x)(u_{j}^\e-u_{i+n}^\e)\le 0\ (\ge 0)\text{ $i=1,2,\ldots,n$}\\

\end{cases}
\end{equation}in distribution sense.
\end{definition}

\begin{lem}\label{lem-contraction}($L^1$-contraction)
Suppose interaction rate $k$ is T-dissipative and $u_i,\ v_i\in W^{1,1}_{loc}(R_+\times \R^n)\cap C(\R_+,W^{1,1}_{loc}(R^n))$ are weak subsolution and supersolution of \eqref{eq-main}   respectively. Then, for all $t_2>t_1$, we have

\begin{equation}
\label{eq-flux}
{\int \displaylimits_{B(0,R)}\sum_{i=1}^{2n}(u_i-v_i)^+dx }\bigg\arrowvert_{t=t_1}^{t=t_2}\le -\int \displaylimits_{t_1}^{t_2}\int \displaylimits_{\partial B(0,R)} \sum_{i=1}^n \left[\fr{(u_i-v_i)^+}{\e}-\fr{(u_{i+n}-v_{i+n})^+}{\e}\right]\nu_i d\sigma_x dt
\end{equation}

\begin{equation}\label{eq-contraction-approx}
\left[\int_{B(0,R)}\sum_{i=1}^{2n}(u_i-v_i)^+dx\right](t_2) \le \left[\int_{B(0,R+\fr{t_2-t_1}{\e})}\sum_{i=1}^{2n}(u_i-v_i)^+dx \right](t_1)
\end{equation}
\begin{equation}\label{eq-contraction}
\left[\int_{\R^n}\sum_{i=1}^{2n}(u_i-v_i)^+dx\right](t_2) \le \left[\int_{\R^n}\sum_{i=1}^{2n}(u_i-v_i)^+dx \right](t_1)
\end{equation}
\end{lem}
\begin{proof}

Since $u_i,\ v_i\in W^{1,1}_{loc}(R_+\times \R^n)\cap C(\R_+,W^{1,1}(R^n))$, we can notice that  integrals above are well defined.
And for  $f(x)\in W^{1,1}_{loc}$, we have $\fr{\partial f^+}{\partial x_i} = sgn^+(f)\cdot \fr{\partial f}{\partial x_i}$. Using this and Definition \ref{def-sol}, we get

\begin{equation} \begin{aligned}
\sum_{i=1}^{2n}\fr{\partial (u_i-v_i)^+}{\partial t}\le -\sum_{i=1}^n \fr{1}{\e}\partial_i[(u_i-v_i)^+-(u_{i+n}-v_{i+n})^+]\\+(A\overrightarrow{u}-A\overrightarrow{v})\cdot sgn^+(\overrightarrow{u}-\overrightarrow{v})\\
\le-\sum_{i=1}^n \fr{1}{\e}\partial_i[(u_i-v_i)^+-(u_{i+n}-v_{i+n})^+]
\end{aligned}\end{equation}
from  Lemma \ref{Lem-Tdissipativity}.
It  holds in sense of the distribution and a.e.
By taking an integral  in the inequality above and applying  the divergence theorem, we will have \eqref{eq-flux}.  \eqref{eq-contraction-approx} comes from finite speed of propagation, $\fr{1}{\e}$. Finally, \eqref{eq-contraction} is given by taking a  limit of \eqref{eq-contraction-approx} as $R\to\infty$.

\end{proof}

\begin{thm}(Comparison Principle)\label{thm-MP}
Under the same assumption with Lemma \ref{lem-contraction}, then we have following comparison theorems:

\begin{enumerate}[(i)]

\item Assume that $[(u_i-u_{i+n})-(v_i-v_{i+n})]\nu_i \ge 0$ a.e. on $\partial B(0,R)\times (t_1,t_2)$ for all i=1,$\ldots$,n (i.e. outward flux of u is greater than that of v in any directions at any points on the boundary). And if  $\ u_i\le v_i\ \ \text{ on } B(0,R)\times\{t_1\}$ for all
$ i=1,\ldots,2n$, then we have
$u_i\le v_i\  \text{ on } B(0,R)\times(t_1,t_2)$
  for all $ i=1,\ldots,2n$.
\item
$ \ u_i\le v_i\ \text{ on } B(0,R+\fr{t_2-t_1}{\e})\times\{t_1\}$ for all
$ i=1,\ldots,2n$  implies that $ \  u_i\le v_i\ \text{ on } B(0,R)\times(t_1,t_2)$  for all
$ i=1,\ldots,2n$.
\item
$\ u_i\le v_i\  \text{ on } \R^n\times\{t_1\}$ for all
$ i=1,\ldots,2n$  implies that $u_i\le v_i\ \text{ on } \R^n\times(t_1,t_2)$  for all
$ i=1,\ldots,2n$.
\end{enumerate}

\begin{proof}
Results (2) and (3) are immediate from Lemma \ref{lem-contraction}. In order to prove (1), it is suffices to show $[(u_i-u_{i+n})-(v_i-v_{i+n})]\nu_i \ge 0$ implies $[(u_i-v_i)^+-(u_{i+n}-v_{i+n})^+]\nu_i\ge0$. One can verify this from case-by-case argument.
\end{proof}

\end{thm}

\begin{remark}
The fact $u_i,\ v_i\in W^{1,1}_{loc}(R_+\times \R^n)\cap C(\R_+,W^{1,1}(R^n))$ means you have well-defined $L^1$ boundary value function from trace theorem. Moreover, you can think of any special time slice of $u_i, v_i$ since they belong to $ C(\R_+,W^{1,1}(R^n))$.
\end{remark}

\begin{remark}
Result (2) and (3) are also true when your weak solution is in the class $u_i \in L_{loc}^\infty(Q) \cap C([0,\infty), L^1_{loc} (\R^n))$ we defined earlier. If initial data is in $W^{1,1}_{loc}(\R^n)\cap L^\infty_{loc}(\R^n)$, your weak solution is in $u_i,\ v_i\in W^{1,1}_{loc}(R_+\times \R^n)\cap C(\R_+,W_{loc}^{1,1}(R^n))$ (see \cite{TL}), and you get the desired result via approximation (density argument) of initial data and employing $L^1$-contraction we proved.
\end{remark}

\section{Flux Estimate via Local Entropy}\label{sec-entropy}

To get convergence, we need a flux estimate , $\int_{0}^{T} \int_\Omega \sum_{i,j} \Jes _{i,j} dxdt $, uniformly in $\e$. This is crucial step and proof of this in various cases will be our main work. However, in case when we have good positive upper and lower bound on initial data (therefore, it is not in $L^1$), we can just mimic the estimates done in \cite{SV}.

\begin{prop}\label{prop-local}
If $\ue_i$ i=1,\ldots,2n are uniformly  bounded above and below, say $0<\fr{1}{M}  \leq \ue_i \leq M <\infty$ on $K\times [0.T]$. Then for all $\phi \in C^\infty_0(\R^n)$ with $supp(\phi) \subset K$, $$\int_0^T \int \sum_{i,j} k(u_{i}^\e,u_{j}^\e,x)^2\Jes_{i,j} \phi^2 dx dt < C$$ for some $C=C(\phi,M,k,T,n)>0$, but independent of $\e$.

\begin{proof}
From the condition of k, there is $A=A(M)>0$ such that $$\fr{1}{A}\le k(u_i,u_j,x) \le A.$$
By multiplying $\fr{1}{\ue_i}\phi^2$ both sides of \eqref{eq-main} and integrating w.r.t. x, we obtain

\begin{equation}\begin{aligned}		
	\partial_t \int \sum_{i=1}^{2n} \log(u_{i}^\e )\phi^2dx + \fr{1}{\e} \int \sum_{i=1}^{n}\Di  \log(\ue_i )-\Di \log(\ue_{i+n}) \phi^2dx \\=\fr{1}{4n\e^2}\int\sum_{i\neq j, 1\leq i,j\leq2n} k(\ue_i,\ue_j,x)(\ue_j-\ue_i)(\fr{1}{\ue_i}-\fr{1}{\ue_j})\phi^2dx,
\end{aligned}\end{equation}
and integration by parts gives us
\begin{equation}\label{eq-entropy-1}
\begin{aligned}		
	\partial_t \int \sum_{i=1}^{2n} \log(u_{i}^\e )\phi^2dx &= \fr{2}{\e} \int \sum_{i=1}^{n}\ ( \log(\ue_i )-\log(\ue_{i+n}))   \phi\Di \phi dx\\&+ \fr{1}{4n}\int\sum_{i\neq j, 1\leq i,j\leq2n} \fr{k(\ue_i,\ue_j,x)(\Je_{j,i})^2}{(\ue_i\ue_j)}\phi^2dx.
\end{aligned}\end{equation}

Meanwhile, \begin{equation} \label{eq-entropy-2}
	\begin{aligned}
		 -\fr{2}{\e} ( \log(\ue_i )-\log(\ue_{i+n}))   \phi\partial_i \phi &= \fr{2\phi\partial_i\phi}{\e}\log(\fr{\ue_{i+n}}{\ue_{i}})\\
		 &\leq \fr{2\phi\partial_i\phi}{1}\fr{\Je_i}{\ue_{i}} \quad (\because \log(1+r) \leq r)\\&
		 \leq \fr{1}{4nA}\fr{J^{\e 2}_i\phi^2}{(\ue_i\ue_{i+n})} + 4nA\fr{\ue_{i+n}}{\ue_{i}}(\partial_i\phi )^2.
		 \end{aligned}
\end{equation}

Using \eqref{eq-entropy-1} and \eqref{eq-entropy-2}, we have
\begin{equation}
	\begin{aligned}
	 \fr{1}{4n}\int\sum_{i\neq j, 1\leq i,j\leq2n} \fr{k(\ue_i,\ue_j,x) \Jes_{j,i} \phi^2}{(\ue_i\ue_j)}dx \leq \fr{1}{4n}\int \sum_{i=1}^{n} \fr{1}{A}\fr{J^{\e 2}_i\phi^2}{(\ue_i\ue_{i+n})}dx   \\+ \int \sum_{i=1}^{n}4nAM^2(\partial_i\phi )^2 dx + \partial_t \int \sum_{i=1}^{2n} \log(u_{i}^\e )\phi^2dx.
	\end{aligned}
	\end{equation}

Therefore, we get
\begin{equation}
	\begin{aligned}
	 \fr{1}{8n}\int^T_0\int\sum_{i\neq j, 1\leq i,j\leq2n} \fr{k(\ue_i,\ue_j,x)  \Jes_{j,i} \phi^2}{(\ue_i\ue_j)}dx \leq C_0(\phi,M,A,T,n).
	\end{aligned}
	\end{equation}
	
Rest of this proof is immadiate if one knows $\fr{1}{A}\leq k(\ue_i,\ue_j,x)  \leq A$ on time interval [0,T].
\end{proof}
\end{prop}

From above uniform local estimate, we get following boundedness of $L^2$ norm of $J_{i,j}$ on $Q_{R,T}$.

\begin{lem}\label{lem-current estimate}
$\ue_i$,i=1,2..,2n be the solution of \eqref{eq-main} with $0<N \leq g_i\leq M$ (hence $N \leq \ue_i \leq M$ by lemma \ref{lem-Apriori}, \ref{lem-Apriori2}). Then for all $R>0$, $\exists C=C(R,M,N,k,T,n)$ independent of $\e$ s.t.
\begin{equation}
\iint_{Q_{R,T}}\sum_{i,j}\Jes_{i,j}dxdt\le C
\end{equation}
This allow us deduce that $\Je_{i,j}$ converges to some $J_{i,j}$ weakly in $L^2_{x,t}(Q_{R,T})$ along some subsequences for all $Q_{R,T}$.

\end{lem}

\begin{remark}
From the same argument, $k(u_{i}^\e,u_{j}^\e,x)\Je_{i,j}$ are also $L^2$ bounded. Thus this has weakly convergent subsequence in $L^2_{loc}(Q)$.
\end{remark}

\begin{remark}
Since $\ue_i\le M$ uniformly in $\e$ and $t$, $\rhoe=\sum_i\ue_i $(subsequence) converge weakly $\rho$ for some $\rho$ in $L^2_{x,t}$ locally.
\end{remark}

\section{Strong Convergence of Diffusive Limit}\label{sec-dl}

We can improve the mode of convergence by a well-known compensated compactness theorem, div-curl lemma \cite{M}, stated below.

\begin{lem}\label{lem-div-curl}(Div-Curl Lemma)  For an open set $A$ of $\R^n$, let $w_\e$ and $ v_\e$ be two sequences such that
\begin{equation}
\begin{aligned}
w_\e\ \rightharpoonup\ w \text{ in $[L^2(A)]^n$-weak},\\
 v_\e\ \rightharpoonup\ v\text{ in $[L^2(A)]^n$-weak},
\end{aligned}
\end{equation}

\begin{equation}
\begin{aligned}
&div(w_\e)\text{ is bounded in }L^2(A) \text{ or compact in $H^{-1}(A)$, \ and}\\
&curl(v_\e)\text{ is bounded in }[L^2(A)]^{n^2} \text{ or compact in $[H^{-1}(A)]^{n^2}$}.
\end{aligned}
\end{equation}
Then we have
\begin{equation}
\langle w_\e,v_\e\rangle\ \longrightarrow \ \langle w,v \rangle
\end{equation}
in distribution sense, where  $\langle \cdot, \cdot \rangle$ is usual inner product.
\end{lem}

In order to use the  div-curl lemma, let us choose $A=Q_{R,T}$ in $\R^{n+1}$, $w^\e=(\Je,\rhoe)$, and $v^\e=(0,\rhoe)$ in $[L^2_{x,t}(Q_{R,T})]^{n+1}$. Then we get
\begin{equation}
\begin{aligned}
div_{x,t}w^\e=div_{x,t}(\Je,\rhoe)=div_x\Je+\partial_t\rhoe =0
\end{aligned}
\end{equation}
by the first equation of \eqref{eq-main3}.  From the definition $(curl\ F )_{ij}=\fr{\partial F_i}{\partial x_j} -\fr{\partial F_j}{\partial x_i}$, we have
\begin{equation}
curl\ v^\e =
\left( \begin{array}{cccc}
0 		& \cdots & 0 & -\partial_{x_{1}}\rhoe \\
\vdots  & \ddots&\vdots & \vdots\\
0 &\cdots&	 0	& -\partial_{x_{n} }\rhoe \\
\partial_{x_{1}}\rhoe& \cdots & \partial_{x_{n}}\rhoe &0
\end{array} \right).
\end{equation}

It remains to check whether the sequences of all entries in this matrix are compact in $H^{-1}(A)$ respectively. We employ the second equation of \eqref{eq-main3} and Lemma \ref{lem-current estimate}. First, we check
\begin{equation}
\begin{aligned}
\| \e^2\Je_i \|_{L^2(S)},\,\,\ \| \rhoe_j-\rhoe_i\|_{L^2(S)}= \| \e \Je_{j,i}\|_{L^2(S)} \to 0\text{ as } \e \to 0.\\
\text{Therefore, }\| \e^2 \partial_t  \Je_i \|_{H^{-1}(S)} ,\,\, \    \|\D_i ( \rhoe_j-\rhoe_i) \|_{H^{-1}(S)} \to \text{ 0 as }\e \to 0.\end{aligned}
\end{equation}
Moreover,  the $L^{\infty}$-bound of $\ue$ implies  the $L^{\infty}$-bound  of $k$ terms in the right hand side of the second equation \eqref{eq-main3}. Then, from  $L^2(A)$-bound of $\Je_{i,j}$, Lemma \ref{lem-current estimate}, it follows  the $k J^\e_{i,j}$ terms in the second equation of \eqref{eq-main3} are  bounded  in $L^2(A)$. Finally, we can notice $\partial_i\rho^\e$ can be expressed by other terms in the second equation of \eqref{eq-main3}  which are either $H^{-1}(A)$ convergent or $L^2(A)$ bounded. Since $L^2(A)$ is compactly embedded in $H^{-1}(A)$, this leads to compactness of $\partial_i\rho^\e$ in $H^{-1}(A)$. We now apply the div-curl lemma to obtain following proposition.

\begin{prop}\label{prop-div-curl app}
Let $\ue_i$, $i=1,\ldots,2n$, be the unique weak solution of \eqref{eq-main} with bounded initial data $g_i$ such that  $0<N \leq g_i\leq M$. Then, on each $A=Q_{R,T}$, we have $\langle v^\e,w^\e\rangle=(\rho^{\e})^{2} \to \rho^2$ in distribution sense along a subsequence. In fact, $\rho^2 \in L^2_{loc}(Q)$ and $(\rhoe)^2 \rightharpoonup \rho^2$ weakly in $L^2_{loc}(Q)$ .
\end{prop}
The weak convergence in the Proposition above comes from $L^\infty$-bound of $\rhoe$ and the uniqueness of its distribution limit.
\begin{corollary}\label{cor-div-curl}

Under the same condition of Proposition \ref{prop-div-curl app}, $$\rhoe \to \rho\text{ and }\ \ue_i \to \fr{\rho}{2n}\text{ in }L^2_{loc}(Q).$$ Furthermore,  $$ k(u_{k}^\e,u_{l}^\e,x) \Je_{i,j} \to k\left(\fr{\rho}{2n},\fr{\rho}{2n},x\right)J_{i,j}\text{ in }L^1_{loc}(Q)\text{ for all }i,j,k,l=1\ldots2n.$$
\begin{proof}
The strong convergences of $\rho^{\e}$ and $u_i^{\e}$ come from Proposition \ref{prop-div-curl app} above and  Lemma 7 of \cite{SV} which can be stated as follows:
 $$\text{If }|A|<\infty,\ \rho^\e \rightharpoonup \rho\text{ and }  (\rho^\e)^2 \rightharpoonup \rho^2\text{ in }L^2(A)\text{, then }\rho^\e \rightarrow \rho\text{ in }L^2(A).$$
Next, the results above and Lemma \ref{lem-current estimate} imply $\ue_i = \fr{1}{2n}(\rhoe + \sum_{k=1}^{2n} \e\Je_{i,k})\to \fr{\rho}{2n}$ in $L^2_{loc}(Q)$. Moreover, by taking a subsequence, we may assume $u^\e_i\to\rho$ almost everywhere. Hence, from the continuity of $k$, we can apply dominated convergence theorem to obtain $$k(u^\e_i,u^\e_j,x)\to k\left(\fr{\rho}{2n},\fr{\rho}{2n},x\right)\text{ in }L^2_{loc}(Q).$$ $L^1_{loc}$ convergence of  $k(u_{k}^\e,u_{l}^\e,x) \Je_{i,j}$ to $k\left(\fr{\rho}{2n},\fr{\rho}{2n},x\right)J_{i,j}$ is obvious from the following standard lemma:
$$\text{If }1<p<\infty\text{, }u_n\to u\text{ in }L^p\text{ and }v_n\rightharpoonup v\text{ in }L^{p'}\text{, then }u_nv_n\to uv\text{ in }L^1$$\end{proof}
\end{corollary}

Now we have  shown that  $\ue \to  \fr{\rho}{2n}$  in $L^2_{loc}$ and $\Je \rightharpoonup J$ in $[L^2_{loc}]^n$ along a subsequence. With theses results, we can  identify equations  satisfied by $\rho$, and $J$.
Since $\rhoe$ and $\Je$ satisfy \eqref{eq-main3} in distribution sense, we have \begin{equation}\label{eq-dl-1}\partial_t \rho + div J = 0\end{equation} in distribution as well by taking a limit. Similarily, in the second equation of \eqref{eq-main3},  uniform $L^2$-bound of $J^\e$,  $L^2$ convergence of $u^\e$ and $L^1$ convergence of  $k(u_{k}^\e,u_{l}^\e,x) \Je_{i,j}$ imply following equation in distribution sense:

\begin{equation}
0 + \D_i\rho =\fr{1}{n} \D_i \sum_{j=1}^{n} 0 +\fr{1}{2n} \sum_{j=1}^{2n}\left[0(\fr{J_{j,i}+J_{j,i+n}}{2})-\left( k\left(\fr{\rho}{2n},\fr{\rho}{2n},x\right)+k\left(\fr{\rho}{2n},\fr{\rho}{2n},x\right) \right) \fr{J_i}{2}\right].
\end{equation}
{\it i.e.}
\begin{equation}\label{eq-dl-2}
 \fr{1}{n}\D_i\rho = -k\left(\fr{\rho}{2n},\fr{\rho}{2n},x\right)J_i \quad \text{ in distribution sense.}
\end{equation}

This implies $\D\rho\in L^2_{loc}(Q)$ and $J=-\fr{1}{n k\left(\fr{\rho}{2n},\fr{\rho}{2n},x\right)} \D\rho\in [L^2_{loc}(Q)]^n$ because k is bounded above and below with positive constants.
 Finally, combining \eqref{eq-dl-2} with \eqref{eq-dl-1}, we get \begin{equation}\label{eq-dl-3}
 \partial_t\rho-\D\cdot\left(\fr{1}{n k\left(\fr{\rho}{2n},\fr{\rho}{2n},x\right)}\D \rho\right) =0 \text{\quad in weak sense.}
 \end{equation}

In the next step, we identity the initial trace of this diffusion equation. Of course, this will be derived by employing the fact that $\{\ue_i\}$ solves \eqref{eq-main} with initial data $\{g_i\}$.
From \eqref{eq-main3} and the definition of weak solution, we get
\begin{equation}
\int_0^\infty\int_{\R^n} \rhoe\partial_t\phi+\langle\Je,\D \phi\rangle dxdt + \int_{\R^n}\sum_{i=1}^{2n} g_i \ \phi dx =0,
\end{equation}
for all $ \phi \in C^\infty(Q)\text{ , which vanishes on }|x|>R\text{ and }t>T \ \text{ for some } R\text{ and } T>0.$
Using Proposition \ref{prop-div-curl app} and Corollary \ref{cor-div-curl}, we can pass to the limit and get the following equation:
\begin{equation}
\int_0^\infty\int_{\R^n} \rho\partial_t\phi-\fr{1}{n k(\fr{\rho}{2n},\fr{\rho}{2n},x)}\langle\D\rho,\D \phi\rangle dxdt + \int_{\R^n}\sum_{i=1}^{2n} g_i \ \phi dx =0  .
\end{equation}

Therefore, $\rho$ solves cauchy problem of following diffusion equation
$$\rho_t=\D\cdot\left( \fr{\D \rho}{n k\left(\fr{\rho}{2n},\fr{\rho}{2n},x\right)}\right) \text{weakly with initial data } \rho(0)=\sum_{i=1}^{2n} g_i.$$

 Regarding the regularity of $\rho$, since the k term has positive upper and lower bounds, this equation is uniformly parabolic. Hence,  $\rho$ is unique for a  given initial data. Moreover, by standard regularity theory and bootstrap argument \cite{L}, it is $C^{\infty}$ on $t>0$ if $k(a,a,x)$ is smooth function of s and x.  The uniqueness shows that the convergence actually takes in any arbitrary sequences. Finally, if $k$ is T-dissipative and initial data is of $L^1$ perturbation of a $C^1$ function having integrable derivatives, we have $C([0,T], L^1_{loc}(\R^n))$, improved mode of convergence.
 
\begin{thm} \label{thm-dl-1}
Suppose initial data $\{g_i\}$ has uniform lower and upper bounds $0<N\le g_i \le M$ for $i=1,\ldots,2n$, then  $\ue_i$, the unique weak solution of \eqref{eq-main}, converge to the same limit $\fr{\rho}{2n}$ in $L^2_{loc}(Q)$ as $\e\to 0$. Here $\rho$ is the unique weak solution of the Cauchy problem \eqref{eq-dl-3} with initial data $\sum_{i=1}^{n} g_i$. 

Moreover, if k has no space dependency, {\it i.e.} $k(a,b,x)=k(a,b)$, and T-dissipative in the sense of Definition \ref{def-dissipativity} and initial data $g_i$ can be written as \begin{equation}\label{eq-compactcondition2}g_i=l_i + h_i,\text{ }h_i\in L^1(\R^n),\text{ and }\ l_i \in C^1(\R^n) \text{ with }\int |Dl_i| < \infty  ,\end{equation} then $\rho^\e \to \rho$ in $C([0,T],L^1_{loc}(\R^n))$.

\begin{proof}
The first part of the theorem above comes from Corollary \ref{cor-div-curl} and the comment after that. It suffices to prove $\rho^\e$ are relatively compact in $C([0,T],L^1_{loc}(\R^n))$.

First, let us assume $g_i \in C^1(\R^n)$ with $|Dg_i|\in L^1(\R^n)$. $L^1$-contraction implies $\sum_i||u_i(x+h)-u_i(x)||_{L^1}\le \sum_i||g_i(x+h)-g_i(x)||_{L^1}$ for all $h\in \R^n$. Thus, $||D\rho^\e||_{L^1(\R^n)}\le\sum_i ||Du_i||_{L^1(\R^n)}\le \sum_i ||Dg_i||$. Hence, $\rho^\e$ are bounded in $L^\infty([0,T], W^{1,1}_{loc})$. Moreover, first equation of eq \eqref{eq-main3} and uniform $L^2_{loc}$ bound of $J^\e$ tell us that $\p_t\rho^\e$ are uniformly bounded in $L^2([0,T], H^{-1}_{loc}(\R^n))$, and hence $\rho^\e(t)\in H^{-1}_{loc}(\R^n)$ are uniformly H\"{o}lder continuous in time. Now, using these time and space regularity, we are able to apply Theorem 5 in \cite{S} and conclude that $\rho^\e$ are relatively compact in $C([0,T], L^1_{loc}(\R^n))$.

  Now if $g_i=l_i+h_i$ as in the statement, let $h^\delta_i$ be usual regularization of $h_i$. {\it i.e.} $h_{i,\delta}(x)= \int \eta_\delta(x-y)h_i(y)dy$. It could be checked that $h_{i,\delta} \to h_i$ in $L^1$ and $|Dh_{i,\delta}|$ are integrable on $\R^n$. If we approximate $g_i$ by $l_i+h_{i,\delta}$, $L^1$-contraction and first result implie $\rho^e$ are also relatively compact in $C([0,T],L^1_{loc}(\R^n))$.
\end{proof}
\end{thm}

\begin{remark}
We have assumed upper and lower bound $0<N\le g_i <M $ in the theorem, but if $k$ has positive lower bounds, we only need an upper bound to obtain the same conclusion. This needs a small change in the proof and see \cite{SV} for the details in 1-D case.
\end{remark}

\section{Local Lower Bound with General Initial Data}\label{sec-general}

From now on, we will consider the case  $k(u,v,x)=k_\alpha(u,v)=(u+v)^\alpha, $ for $\ -1\le \alpha \le 1$. In this case, our target diffusion equation will be Porous Medium Equation (PME) or Fast Diffusion Equation (FDE)   $\partial_t\rho-\fr{1}{ (1-\alpha)n^{1-\alpha}}\Delta \rho^{1-\alpha} =0$, which  has been studied extensively to model nonlinear diffusion in  applications to physics, engineering and geometry. Degenerate or singular  diffusion take places only if the value of the  solution touches zero or becomes infinity. Otherwise, it is uniformly parabolic and the behavior of a solution is very similar to that of linear heat equation. Hence, it is natural to ask whether we could obtain such kind of solution from our kinetic model through diffusive limit, {\it i.e.} diffusive limit of \eqref{eq-main} from unbounded or degenerate initial data is one of our main interest. We want to understand the minimal condition for the convergence of diffusive limit and compare it with the condition for the existence theory of target diffusion equations.

Our analysis on diffusive limit of \eqref{eq-main} will be divided into several cases according to different ranges of $\alpha$. It is due to the fact that the key characteristic of diffusion changes with $\alpha$.  We would like to  recall some known facts.  In PME range $-1\le \alpha<0$ ,  every $L^1$ initial data generates the unique mass preserving global solution. Moreover, it has finite speed of propagation property and thus creates free boundary if the initial data is compactly supported. Meanwhile, FDE range $0<\alpha\le1$, has more delicate phenomenon. In fact, this range could be again divided into supercritical and subcritical ranges. In supercritical range $0<\alpha<\fr{2}{n}$, theory is quite similar to PME compared to that of subcritical equation. However, this equation  is singular when $\rho=0$ and there is lack of some smoothing estimates. Finally, in critical and subcritical range $\fr{2}{n}\le\alpha\le1$, no solution can be driven from dirac mass since diffusion is too slow at that point. Also, there is no global existence on a very large class of initial data. Rather than its global existence, it shows finite time extinction property. For instance, every solution with $L^1$ initial data vanishes in finite time. Finally, in the case  $\alpha=1$, so called logarithmic fast diffusion equation, a solution is no longer unique. The existence, uniqueness, regularity, decaying, extinction, and asymptotic behavior of PME and FDE have been studied for decades and more details can be found in \cite{DK}, \cite{V}, and \cite{V2} . More recent results on asymptotic behaviors and their convergence rates can be found in \cite{BBDGV}, \cite{BDGV}, \cite{DKS}, \cite{DS}, \cite{DS2}, \cite{H}, \cite{H2}, \cite{HK}, and \cite{HK2}.\\

Resolution of the problem with general $L^1$ initial data in one space dimension is presented in Section 8, \cite{SV}. They approximate initial data of original problem with $\delta$-lifted bounded initial data, which has positive upper and lower bounds and whose diffusive limit has been proved at Section \ref{sec-dl}. Such truncation does not make serious difficulties because of the $L^1$-contraction property of the solution, where we can always approximate the original data in $L^1$ sense. However, $\delta$-lifting creates nontrivial issue due to the fact that the difference with the original data is no longer in $L^1$.

 To obtain the convergence of diffusive limit with original data, they needed to show difference between  the original solution and $\delta$-lifted solution of \eqref{eq-main} is locally arbitrary small as $\delta \to 0$ with $\e$-independent $L^1$-norm . It was crucial to prove the following: If \begin{equation}\label{eq-general-1}\int_{\R^n}\sum g_i(1+|x|^2)^\fr{q}{2}dx < C <\infty\end{equation} at $t=0$, then  \begin{equation}\label{eq-general-2} \int_{\R^n} \sum u_i^\e(1+|x|^2)^\fr{q}{2}dx, \int_{\R^n} \sum( u^\e_{\delta,i}-\delta)(1+|x|^2)^\fr{q}{2}dx < \overline{ M}\end{equation} for all $t\in [0,T]$ for some $\overline{M}=\overline{M}(C,n,\alpha,T)$ where $ u_{\delta,i}^\e$ is the solution of \eqref{eq-main} with initial data $g_i+\delta$. \eqref{eq-general-2}  tells us mass concentration over the time. On the other hand,  since mass $\int \rho^\e$ and $\int( \rho^\e_\delta - 2n\delta)$ are equal and preserved, $\rho^\e$ and $\rho^\e_\delta - 2n\delta$ will have almost equal mass in very large domains with a small error independent of $\e$ and $\delta$. This estimate says mass does not escape to infinity as $\e\to 0$ thus this is a kind local version of mass conservation. Since $\rho^\e_\delta$, solution with $\delta$-lifted initial data, is located above $\rho^\e$, we can control local $L^1$ difference between these two solution uniformly in $\e$ by simply taking larger domain. With this local estimate, \cite{SV} proves diffusive limit for $n=1$ and $|\alpha|\le1$ cases.

 However, if the target diffusion equation has no mass conservation, it is unnatural to expect similar estimates. The fact is, in one dimension and $|\alpha|\le 1$ range, it does not fall in subcritical exponent range. Thus, there always exist mass conserved solutions. In higher dimensional case, it is not hard to check the method of \cite{SV} can be used again in supercritical exponent $-1\le\alpha<\fr{2}{n}$, where mass conservation holds. To be speciic, one can reproduce Appendix II of \cite{SV} in this higher dimension case if the exponent is supercritical. On the other hand, our method in this paper for subcritical range can be also applied even to the supercritical case, which was proved by using \eqref{eq-general-2} for one dimension \cite{SV}.\\

Let us explain our steps in higher dimension subcritical exponent case. We recall that local uniform upper and lower bounds of $u^\e_i$s give us the  local flux estimates, div-curl lemma, and convergence theorem in Section \ref{sec-dl}. In this section, we consider more general initial data with a  decay condition at infinity, $g_i \sim \fr{1}{ |x|^\fr{2}{\alpha}}$. We prove the tail behavior of $u^\e$ will be preserved in time uniformly in $\e$, thus we obtain local uniform lower bound. And then, we consider larger admissible class of initial data allowing $L^1$ perturbation from given function $g_i$. If $\alpha$ is in supercritical range, it should be noticed that this tail, $|x|^{-\fr{2}{\alpha}}$, is $L^1$ integrable as Barenblatt profile is integrable on solution of PME in this range. It allows us to approximate any $L^1$ initial data and recover the  same result as \cite{SV}. In subcritical range, the decay profile of the initial data at infinity is usually assumed  in analysis of FDE. 
For instance, we have $L^1$-contraction property in the class of solutions satisfying the decay rate, $|x|^{-\fr{2}{\alpha}}$. Moreover, those results in \cite{BBDGV}, \cite{BDGV}, \cite{DKS}, \cite{DS}, \cite{DS2}, \cite{H}, \cite{H2}, \cite{HK}, \cite{HK2} assume their initial profiles are trapped between two Barenblatt profiles, $L^1$-perturbation from them, or decay rate slower than $|x|^{-\fr{2}{\alpha}}$.
 Finally with a similar idea, we can also control local uniform upper bounds of solutions which will allow us to consider initial data having some growth condition beyond $L^\infty$-bound.
\\

Now we are going to show local uniform lower bound of equation in critical and subcritical exponents. Since we have the comparison principle, we try to find lower bound using barrier argument. Unlike from PME, FDE, and other diffusion equations where numerous examples of exact solutions are known from self-similarity, it is not trivial to find an exact solution, or even subsolution of this system \eqref{eq-main}. We couldn't find any spatially global explicit subsolution, but could found a local explicit one. Roughly, this local subsolution will be used to make local lower bound by following steps below.

\begin{enumerate}
\item Assume $g_i \ge C\cdot f_0$ for some nonnegative function $f_0$ with $C>1$, which we know explicit subsolution wof target diffusion equation with this initial data $f_0$.
\item We will find a solution $f(t)$ of target diffusion equation \eqref{eq-fastdiff} with initial data $f(0)=f_0$, and then perturb $f$ to get $\{f^{\e,R,T}_i\}$ which is subsolution of \eqref{eq-main} at least on $B(0,R+\fr{T}{\e})\times (0,T)$. We will  attain $\{f^{\e,R,T}_i\}$ by adding the limit profile  $\{\fr{
f}{2n}\}$ with  correcting terms of order $\e$ and $\e^2$.
\item We show $f^{\e,R,T}_i(0) \le g_i$ on $B(0,R+\fr{T}{\e})$. This proves $u^\e_i \ge f^{\e,R,T}_i$ on $B(0,R)\times (0,T)=Q_R^T$ by comparison principle, Theorem \ref{thm-MP} (ii).
\item Find a function g of the form $c\cdot f$ such that $g\le  f^{\e,R,T}$ on $Q_R^T$ for all $\e, \ R >0$. This shows $\ue_i\ge g>0 $ independently with $\e$.
\item We can use local entropy estimates and  find hydrodynamical limit for a short time (0,T), which is a solution of the limit equation \eqref{eq-fastdiff}.
\item Now we increase T as much as possible.
\end{enumerate}\mbox{}

\subsection{Formal Asymptotic Expansion}\mbox{}\\

As mentioned before, previous research on this problem in 1-D case was done in semi-group framework, \cite{K}\cite{F2}. They used semi-group convergence theorem in order to approximate semi-group of the limit equation from semi-group of Carleman equation. It was essential to find series of initial data $u^\e_i$ which converge  to $u_i$ and $\lim_{\e \to 0}A^\e(\{u^\e_i\})=A(\{u_i\})$ in some sense. In previous statement, $A^\e$ is semi-group generator of Carleman System and $A$ is generator of its limit  equation. To do so, a kind of first order asymptotic expansion of \eqref{eq-main} from $\fr{\rho}{2}$ was employed. To construct subsolution, we try this ansatz again. However, it turns out that the first order expansion is not enough to make proper approximation, so we consider  the second order approximation.

Let us express $u_i, i=1,\ldots,2n$, as
\begin{equation}\label{eq-general-0-1}u_i=\fr{1}{2n}(\rho+A_i\e+B_i\e^2)
\end{equation}
where $A_i$ and $B_i$ are functions of $(x,t)$. We want to find certain relation between $A_i$ and $ B_i$ so that \eqref{eq-general-0-1} is an approximated solution \eqref{eq-main} upto order $\e$.

To compute $\partial_t u_i + \fr{1}{\e}v_i\cdot Du_i -\fr{1}{2n\e^2}\sum_j(u_i+u_j)^\alpha(u_j-u_i)$, we need the following Taylor's expansion.

\begin{lem} \label{lem-taylor}
Let $f(x)=(1+ax+bx^2)^\alpha$ where $\alpha \in (0,1]$. Then,
\begin{equation}\label{eq-remainder}\begin{aligned}& f(x)=1+\alpha ax+\binom{\alpha}{2}(ax)^2+\alpha bx^2+R_{a,b}(x)x^3 \quad{ and}\\ &R_{a,b}(x)=\left[\binom{\alpha}{3}(1+a\zeta+b\zeta^2)^{\alpha-3}(a+2b\zeta)^3+2\binom{\alpha}{2}(1+a\zeta+b\zeta^2)^{\alpha-2}(a+2b\zeta)b\right]\end{aligned}\end{equation}
for some $\zeta$ between 0 and x.
\begin{proof}It comes from the standard remainder estimate in Taylor's Theorem.
\end{proof}
\end{lem}

By applying the lemma above in $(u_i+u_j)^\alpha=\left(\fr{\rho}{n}\right)^\alpha(1+\fr{A_i+A_j}{2\rho}\e+\fr{B_i+B_j}{2\rho}\e^2)^\alpha$, we have
\begin{equation}
\begin{aligned}
&\fr{1}{2n\e^2}(u_i+u_j)^\alpha (u_i-u_j)=\fr{1}{(2n)^2}\left(\fr{\rho}{n}\right)^\alpha\left[\fr{1}{\e}(A_i-A_j)+\left\{\fr{\alpha}{2\rho}(A_i^2-A_j^2) +B_i-B_j\right\}\right.\\
& +\e\left\{\fr{\alpha}{\rho}(A_iB_i-A_jB_j)+\binom{\alpha}{2}\left(\fr{1}{2\rho}\right)^2(A_i^3+A_i^2A_j-A_iA_j^2-A_j^3) \right\} \\
&+ \e^2 \left\{\fr{\alpha}{2\rho}(B_i^2-B_j^2)+ \binom{\alpha}{2}\left(\fr{1}{2\rho}\right)^2(A_i+A_j)^2(B_i-B_j)+\left(  R_{\fr{A_i+A_j}{2\rho},\fr{B_i+B_j}{2\rho}}(\e)\right)(A_i-A_j)\right\}\\
&\left. +\e^3\left\{ \left(  R_{\fr{A_i+A_j}{2\rho},\fr{B_i+B_j}{2\rho}}(\e)\right)(B_i-B_j) \right\}\right],
\end{aligned}
\end{equation}

and we expect from \eqref{eq-main}
\begin{equation}\label{eq-expansion}
\begin{aligned}
&\partial_t u_i + \fr{1}{\e} D_{v_i}u  -\fr{1}{2n\e^2}\sum_j(u_i+u_j)^\alpha(u_j-u_i)=\fr{1}{2n\e}\left[ D_{v_i}\rho+\fr{1}{2n}\left(\fr{\rho}{n}\right)^\alpha(2nA_i-\sum_jA_j)\right]\\&
+\fr{1}{2n}\left[\rho_t+D_{v_i}A_i +\fr{1}{2n}\left(\fr{\rho}{n}\right)^\alpha(2nB_i-\sum_jB_j+\fr{\alpha}{2\rho}(2nA_i^2-\sum_jA_j^2))\right] \\&
+\fr{\e}{2n}\left[\partial_tA_i+D_{v_i}B_i+\fr{1}{2n}\left(\fr{\rho}{n}\right)^\alpha \sum_j \left\{ \fr{\alpha}{\rho } (A_iB_i-A_jB_j)+\left(\fr{1}{2\rho}\right)^2\binom{\alpha}{2}\left((A_i^3+A_i^2A_j-A_iA_j^2-A_j^3)\right)\right\}\right]\\
&
+\fr{\e^2}{2n}\left[\partial_t B_i +\fr{1}{2n}\left(\fr{\rho}{n}\right)^\alpha\sum_j\left\{\fr{\alpha}{2\rho}(B_i^2-B_j^2)+ \binom{\alpha}{2}\left(\fr{1}{2\rho}\right)^2(A_i+A_j)^2(B_i-B_j)+\left(  R_{\fr{A_i+A_j}{2\rho},\fr{B_i+B_j}{2\rho}}(\e)\right)(A_i-A_j)\right\}\right]\\
\end{aligned}
\end{equation}

$$\begin{aligned}
& +\fr{\e^3}{2n}\left[\fr{1}{2n}\left(\fr{\rho}{n}\right)^\alpha\sum_j \left(  R_{\fr{A_i+A_j}{2\rho},\fr{B_i+B_j}{2\rho}}(\e)\right)(B_i-B_j)  \right]  \\
&=O(\e).    
\end{aligned}$$

 Since \eqref{eq-expansion} holds for each $\e$, we have
 \begin{equation}\label{eq-general-0-2}
 \fr{1}{2n\e}\left[ D_{v_i}\rho+\fr{1}{2n}\left(\fr{\rho}{n}\right)^\alpha(2nA_i-\sum_jA_j)\right]=0
 \end{equation} and
  \begin{equation}\label{eq-general-0-3}\fr{1}{2n}\left[\rho_t+D_{v_i}A_i +\fr{1}{2n}\left(\fr{\rho}{n}\right)^\alpha(2nB_i-\sum_jB_j+\fr{\alpha}{2\rho}(2nA_i^2-\sum_jA_j^2))\right]=0.
 \end{equation}
  From natural symmetries $A_i=-A_{i+n}$, we get $\sum_j A_j =0$ and $A_i=-\left(\fr{n}{\rho}\right)^\alpha D_{v_i}\rho$ from \eqref{eq-general-0-2}. Next, if we plug these $A_i$s into \eqref{eq-general-0-3}, we find \eqref{eq-general-0-3} for $i$ and $i+n$ will be equal. So, we could assume $B_i=B_{i+n}$. Using $\rho _t = \D \cdot(\fr{1}{n\left(\fr{\rho}{n}\right)^\alpha}\D \rho)$, we eventually get $B_i=B_{i+n}=\left(\fr{n}{\rho}\right)^{\alpha}\partial_i\left(\left(\fr{n}{\rho}\right)^{\alpha}\partial_i\rho\right)-\fr{\alpha}{2\rho}\left(\left(\fr{n}{\rho}\right)^\alpha \partial_i\rho \right)^2$.

From now on, for any given $\rho$, we define
\begin{equation}\label{eq-AB}
\begin{aligned}
&A_i=-\left(\fr{n}{\rho}\right)^\alpha D_{v_i}\rho=\begin{cases}-\left(\fr{n}{\rho}\right)^\alpha \partial_i\rho, \quad 1\le i \le n \\ \left(\fr{n}{\rho}\right)^\alpha \partial_{i-n}\rho,\quad n+1\le i \le 2n \end{cases}\text{ and } \\ &B_i=B_{i+n}=\left(\fr{n}{\rho}\right)^{\alpha}\partial_i\left(\left(\fr{n}{\rho}\right)^{\alpha}\partial_i\rho\right)-\fr{\alpha}{2\rho}\left(\left(\fr{n}{\rho}\right)^\alpha \partial_i\rho \right)^2 \quad 1\le i \le n.
\end{aligned}
\end{equation} and it is worth noting that \eqref{eq-general-0-3} becomes $$\fr{1}{2n}\left[\partial_t\rho-\D\cdot\left(\fr{1}{n^{1-\alpha} \cdot \rho^\alpha}\D \rho\right)\right]$$ under these definitions.
{
\newline

\subsection{Local Subsolutions for $0\le 1-\alpha< \fr{n-2}{n}$ }\label{subsec-general-2}\mbox{}\\

As mentioned before, local subsolutions could be found for all $0\le 1-\alpha \le 2$. However, they will be constructed differently because self-similar solutions of  Barenblatt type are written in different form. We illustrate the tricky case $0< 1-\alpha< \fr{n-2}{n}$.  Other cases will be explained in Section \ref{sec-general2}.\\

One typical self-similar solution of  \begin{equation}\label{eq-fastdiff}\rho _t = \D \cdot(\fr{1}{n^{1-\alpha}\rho^\alpha}\D \rho)\end{equation} for  $0\le 1-\alpha< \fr{n-2}{n}$ is $$\rho=C_\alpha \left( \fr{T-t}{|x|^2}\right)^\fr{1}{\alpha}\text{, where }(C_\alpha)^\alpha=2n^{\alpha-1}\left(n-\fr{2}{\alpha}\right).$$

On the other hand, our subsolution should lie below a given initial data, which requires its $L^\infty$ bound. Thus, we modify $\rho$ to be a bounded subsolution.

\begin{prop}\label{prop-barrier}
If $n\ge3$, $0\le 1-\alpha< \fr{n-2}{n}$, and $R>0$,\begin{equation}\label{eq-barrier}
\Psi_{\alpha,n,R,T}(x,t)=C_\alpha \left(\fr{T-ct}{|x|^2+R^2}\right)^\fr{1}{\alpha}, \quad c-1>\fr{2}{\alpha}\fr{1}{n-\fr{2}{\alpha}}
\end{equation} is explicit example of bounded subsolution and
\begin{equation}
\Psi_t-\D \cdot(\fr{1}{n^{1-\alpha}\Psi^\alpha}\D \Psi)\le -\left((c-1)\left(n-\fr{2}{\alpha}\right)-\fr{2}{\alpha}\right)\fr{2}{\alpha}\fr{\Psi^{1-\alpha}}{n^{1-\alpha}}\fr{1}{|x|^2+R^2}.
\end{equation}
\begin{proof}First, notice that

$$
\begin{aligned}
&\Psi_t=\fr{-c}{\alpha(T-ct)}\Psi=C_\alpha^\alpha\fr{-c}{\alpha(|x|^2+R^2)}\Psi^{1-\alpha},\\
&\Psi_i=-\fr{2}{\alpha}\fr{x_i}{|x|^2+R^2}\Psi,\text{ and}\\
&\Psi_{ii}=\left(\fr{2}{\alpha}\right)^2\fr{x_i^2}{(|x|^2+R^2)^2}\Psi-\fr{2}{\alpha}\fr{1}{|x|^2+R^2}\left(1-\fr{2x_i^2}{|x|^2+R^2}\right)\Psi.
\end{aligned}
$$

We now we obtain$$\Psi_t-\D \cdot(\fr{1}{n^{1-\alpha}\Psi^\alpha}\D \Psi)=\fr{2\Psi^{1-\alpha}}{\alpha n^{1-\alpha}(|x|^2+R^2)^2}\left[
\fr{2}{\alpha}R^2-(c-1)\left(n-\fr{2}{\alpha}\right)(|x|^2+R^2)\right]<0.$$
\end{proof}
\end{prop}
We will use the notation $\Psi$ instead of $\Psi_{\alpha,n,R,T}$ if there is no confusion and constant $c$ is fixed throughout the paper. This strict subsolution $\Psi(x,t)$ will be used to compensate error terms of order $\e$ and $\e^2$ on the RHS of \eqref{eq-expansion}.

\begin{prop}\label{prop-general-1}
For $\Psi(x,t)$, there is $c_3=c_3(c,n,\alpha)>0$ such that LHS of \eqref{eq-expansion} for $\Psi$ is negative for all $ i=1,\ldots,2n$ on $(|x|^2+R^2)^\fr{1}{2}<\fr{c_3}{\e}(T-tc)$ and $t\in[0,\fr{T}{c})$.
\begin{proof}All constants $c_i$ below will depend only on $n,\alpha,c$.

By plugging  \eqref{eq-barrier} in the definitions of $A_i$ and $B_i$, we get
\begin{equation}\label{eq-general-3}\begin{aligned}
A_i&=\fr{2n^\alpha}{\alpha}\fr{x\cdot v_i}{|x|^2+R^2}\Psi^{1-\alpha}\quad \text{ and } \\
B_i&=\fr{2n^\alpha}{\alpha}\left( \left(\fr{2}{\alpha}-1\right)\fr{x_i^2}{|x|^2+R^2}-1\right)\fr{\Psi^{1-2\alpha}}{|x|^2+R^2}.
\end{aligned}
\end{equation}
By using  Proposition \ref{prop-barrier} and the expansion of \eqref{eq-expansion},
\begin{equation}\begin{aligned}
\text{LHS of }\eqref{eq-expansion}\le \fr{1}{\e}\cdot0  - \left((c-1)\left(n-\fr{2}{\alpha}\right)-\fr{2}{\alpha}\right)\fr{2}{\alpha}\fr{\Psi^{1-\alpha}}{n^{1-\alpha}}\fr{1}{|x|^2+R^2}\\ +\text{ higher order terms.}
\end{aligned}
\end{equation}

Since \begin{equation}\label{eq-relation}\fr{\Psi}{T-ct}=(C_\alpha)^\alpha\fr{\Psi^{1-\alpha}}{|x|^2+R^2},\end{equation}this could be rewritten
\begin{equation}
\label{eq-general-4-1}
\begin{aligned}
\text{LHS of }\eqref{eq-expansion}\le -c_1\fr{\Psi}{T-ct}+\text{higher order terms},\quad c_1>0.
\end{aligned}
\end{equation}

In order to control higher order terms, we need some estimates on $A_i$ and $B_i$.
The followings are not hard to check from \eqref{eq-general-3} and \eqref{eq-relation}:
\begin{equation}\label{eq-general-4-2}\begin{aligned}
|A_i|&\le c_0\fr{\Psi}{T-ct}(|x|^2+R^2)^\fr{1}{2},\quad &&|\partial_tA_i|\le c_0 \fr{\Psi}{(T-ct)^2}(|x|^2+R^2)^\fr{1}{2}\\
|B_i|&\le c_0\fr{\Psi}{(T-ct)^2}(|x|^2+R^2),\quad &&|D_{v_i}B_i|\le c_0\fr{\Psi}{(T-ct)^2}(|x|^2+R^2)^\fr{1}{2}\\
|\partial_t B_i| &\le c_0 \fr{\Psi}{(T-ct)^3}(|x|^2+R^2),\quad&& |\Psi^\alpha A_i|\le c_0 \fr{\Psi}{(|x|^2+R^2)^\fr{1}{2}}\\
|\Psi^\alpha B_i|&\le c_0 \fr{\Psi}{(T-ct)}  \quad &&\text{ and so on.}
\end{aligned}
\end{equation}
\\
Therefore, $\e$-order terms in \eqref{eq-expansion} will satisfy \begin{equation}\begin{aligned}
\fr{\e}{2n}&\left[\partial_tA_i  +D_{v_i}B_i+\fr{\alpha}{2\Psi n} \left(\fr{\Psi}{n}\right)^\alpha(2nA_iB_i-\sum_jA_jB_j)\right. \\ &\left.+\fr{1}{8\Psi^2 n}\binom{\alpha}{2}\left(\fr{\Psi}{n}\right)^\alpha(\sum_j(A_i^3+A_i^2A_j-A_iA_j^2-A_j^3))\right]\\
&\le c_2\fr{\Psi}{T-ct}\fr{(|x|^2+R^2)^{\fr{1}{2}}\e}{T-ct} \quad \text{by \eqref{eq-general-4-2}}
\end{aligned}
\end{equation}
Next, we need to control remainder $R$. For this, let us assume
\begin{equation}
\fr{(|x|^2+R^2)^\fr{1}{2}\e}{T-ct}\le c_{3}.
\end{equation}Thus, we have $$\begin{aligned}
&a:=\left|\fr{A_i+A_j}{2\Psi}\right|\le c_0\left(\fr{(|x|^2+R^2)^{\fr{1}{2}}}{T-ct}\right)\le\fr{c_3c_0}{\e}\\
&b:=\left|\fr{B_i+B_j}{2\Psi}\right|\le c_0\left(\fr{(|x|^2+R^2)^{\fr{1}{2}}}{T-ct}\right)^2\le \fr{c_3^2c_0}{\e^2}. 
\end{aligned} $$
By taking $c_3$ small enough so that $(c_{3}+c_{3}^2)c_0<\fr{1}{2}$,
a bound on remainder terms $R_{\fr{A_i+A_j}{2\Psi},\fr{B_i+B_j}{2\Psi}}(\e)$ in equation \eqref{eq-expansion} could be obtained from \eqref{eq-remainder} as follows:
\begin{equation}\label{eq-remainderestimate}
\begin{aligned}
|R(\e)|&\le \left[\binom{\alpha}{3}\left(\fr{3}{2}\right)^{\alpha-3}(|a|+2|b\zeta|)^3+2\binom{\alpha}{2}\left(\fr{3}{2}\right)^{\alpha-1}(|a|+2|b\zeta|)|b|\right]\\
&\le \fr{1}{\e} \left[\binom{\alpha}{3}\left(\fr{3}{2}\right)^{\alpha-3}(|a|+2|b\zeta|)^2+2\binom{\alpha}{2}\left(\fr{3}{2}\right)^{\alpha-1}|b|\right]\\
&\le c_4 \left[\fr{1}{\e}\left(\fr{(|x|^2+R^2)^{\fr{1}{2}}}{T-ct}\right)^2+ \left(\fr{(|x|^2+R^2)^{\fr{1}{2}}}{T-ct}\right)^3+\e\left(\fr{(|x|^2+R^2)^{\fr{1}{2}}}{T-ct}\right)^4\right]
\end{aligned}
\end{equation}
and now we can control $\e^2$ and higher order terms in \eqref{eq-expansion} using \eqref{eq-relation},\eqref{eq-general-4-2} and \eqref{eq-remainderestimate} :
\begin{equation}
\begin{aligned}
&\fr{\e^2}{2n}\left[\partial_t B_i +\fr{1}{2n}\left(\fr{\Psi}{n}\right)^\alpha\sum_j\left\{\fr{\alpha}{2\Psi}(B_i^2-B_j^2)+ \binom{\alpha}{2}\left(\fr{1}{2\Psi}\right)^2(A_i+A_j)^2(B_i-B_j)+ R(\e)(A_i-A_j)\right\}\right]\\
&\le c_5\fr{\Psi}{T-ct}\left[\left(\fr{(|x|^2+R^2)^{\fr{1}{2}}\e}{T-ct}\right)+ \left(\fr{(|x|^2+R^2)^{\fr{1}{2}}\e}{T-ct}\right)^2+\left(\fr{(|x|^2+R^2)^{\fr{1}{2}}\e}{T-ct}\right)^3\right]\\
&\text{\qquad and}\\
&\fr{\e^3}{2n}\left[\fr{1}{2n}\left(\fr{\Psi}{n}\right)^\alpha\sum_j  R(\e)(B_i-B_j)  \right]\\
&\le c_5\fr{\Psi}{T-ct}\left[\left(\fr{(|x|^2+R^2)^{\fr{1}{2}}\e}{T-ct}\right)^2+ \left(\fr{(|x|^2+R^2)^{\fr{1}{2}}\e}{T-ct}\right)^3+\left(\fr{(|x|^2+R^2)^{\fr{1}{2}}\e}{T-ct}\right)^4\right].
\end{aligned}
\end{equation}

In conclusion, there is some positive constant $c_3$ such that
\begin{equation}\begin{aligned}
\text{LHS of }\eqref{eq-expansion}\le  \left[-c_1+(c_2+c_5)c_3+2c_5c_3^2 + 2c_5c_3^3+c_5c_3^4\right]\fr{\Psi}{T-ct}
\end{aligned}
\end{equation} on $(x,t)$ satisfying $\fr{(|x|^2+R^2)^\fr{1}{2}\e}{T-ct}< c_{3}$. 

Finally, we get the desired inequality by changing $c_3$ into smaller one so that $$-c_1+(c_2+c_5)c_3+2c_5c_3^2 + 2c_5c_3^3+c_5c_3^4<0.$$ 
\end{proof}

\begin{remark} \label{remark-general-1}
We can notice the proof of Proposition \ref{prop-general-1} will not be applied if $\rho$ decays faster  than $\left(\fr{1}{|x|^2}\right)^\fr{1}{\alpha}$. The $\e$ and $\e ^2$ order terms in RHS of \eqref{eq-expansion} can't be controlled by $0$-th order term on the region $(|x|^2+R^2)^\fr{1}{2}<\fr{c_4}{\e}(T-tc)$ for some fixed $c_4$ (independent of $\epsilon$). In fact, if $\rho$ has faster decay, $c_4$ should get smaller as $\e \to 0$.
\end{remark}

\end{prop}

\begin{prop}\label{prop-general-2}
In above setting, if $\overline{ u_{i}}$ are defined as \eqref{eq-general-0-1} using $\Psi$ in place of $\rho$, there exists $c_6=c_6(c,n,\alpha)>0$ {\it s.t.} $$\fr{1}{4n}\Psi \le \overline{ u_i^\e} \le \fr{3}{4n}\Psi\quad\text{ for all }  i=1,\ldots,2n$$ on $(|x|^2+R^2)^\fr{1}{2}<\fr{c_6}{\e}(T-tc)$ and $t\in[0,\fr{T}{c})$.
\begin{proof}
It follows from a simple computation with  \eqref{eq-general-4-2}\end{proof}
\end{prop}

By combining Proposition \ref{prop-general-1} and \ref{prop-general-2}, we have our desired family of local solutions.
\begin{corollary}\label{cor-general-1}
For all $\e>0$ and $\Psi_{\alpha,n,R,T}$ of \eqref{eq-barrier}, there exists a family $\{{\overline u_i^\e}\}_{i=1}^{2n}$ which is a subsolution of \eqref{eq-main} satisfying $\fr{1}{4n}\Psi \le \overline{ u_i^\e} \le \fr{3}{4n}\Psi$ on $(|x|^2+R^2)^\fr{1}{2}<\fr{\overline{ c}}{\e}(T-tc)$ and $t\in[0,\fr{T}{c})$. Constants $c$ and $\overline{c}$ depend on $n $ and $\alpha$.
\end{corollary}

\section{Diffusive limit with general initial data when $k=k_\alpha$}\label{sec-general2}

When the interaction rate is of our concern $k=k_\alpha$, we now prove diffusive limit convergence of \eqref{eq-main} toward \eqref{eq-target} with unbounded and degenerate initial data. The final goal is to prove the main result of this paper, Theorem \ref{thm-mmainthm}. As it is pointed, we show local positive bounds by using barriers obtained in Section \ref{sec-general}. Then, the rest of the proof would be a consequence of theory in Section \ref{sec-dl}. First of all, for these unbounded solutions, we need a generalized notion of solution \eqref{eq-main}.

\begin{definition}\label{def-mild} (Mild solution of \eqref{eq-main})
$\{u_i\}_{i=1}^{2n}\in C([0,T], L^1_{loc}(\R^n))$ is called mild solution of \eqref{eq-main} if this is $C([0,T],L^1(\R^n))$ limit of some weak solution of \eqref{eq-main} $\{u_{i}\}_n$.
\end{definition}

It is obvious from the definition and the $L^1$-contraction property of $k_\alpha$, $|\alpha|\le1$ that there exists a unique  mild solution when the initial data is $L^1$-perturbation of $L^\infty_{loc}$. {\it i.e.} there exist $f_i$ in $L^\infty_{loc}$ such that $f_i-g_i \in L^1(\R^n)$. Moreover, these solutions also have $L^1$-contraction property if two solutions are $L^1$ close initially.

Next, we need a local subsolution of Section \ref{subsec-general-2} for all ranges of $\alpha$. Thus let us introduce $\Psi$ in these ranges.

\begin{definition}\label{def-Psi}
For $n \ge 2$ and $\alpha\in [-1,1]$, let us define $\Psi_{\alpha,n,R,T}=\Psi_{R,T}$ as follows: ({\it c.f.} Proposition \ref{prop-barrier})
\begin{enumerate}

\item  If $\alpha\in (\fr{2}{n},1]$ and $n\ge3$,

$\Psi_{R,T}(x,t)=C_\alpha \left(\fr{T-ct}{|x|^2+R^2}\right)^\fr{1}{\alpha}, \ (C_\alpha)^\alpha=2n^{\alpha-1}\left(n-\fr{2}{\alpha}\right)\quad c-1>\fr{2}{\alpha}\fr{1}{n-\fr{2}{\alpha}}.$\\

\item  If $\alpha=\fr{2}{n}$ and $n\ge2$,

$\Psi_{R,T}(x,t)=\left(\fr{1}{n^{1-\alpha}}\right)^\fr{1}{\alpha}\left(\fr{T}{ \left(|x|^2+R^2e^{\fr{4nt}{T}}\right)(\fr{t}{T}+1) }  \right)^\fr{1}{\alpha}.$\\

\item  If $\alpha\in (0,\fr{2}{n})$ and $n\ge2$,

$\Psi_{R,T}={C}_\alpha \left(\fr{T}{|x|^2+R^2(t+T)^{\fr{4}{2-n\alpha}}}\right)^\fr{1}{\alpha},\quad ({C}_\alpha)^\alpha = 2n^{\alpha-1}\left(\fr{2}{\alpha}-n\right).$\\

\item  If $\alpha=0$ and $n\ge2$,

$\Psi_{R,T} =  \fr{R^2}{\left(4\pi {(t+T)}\right)^{\fr{n}{2}\overline{{c}}}}e^{-\fr{\hat{c}n|x|^2}{4(t+T)}},\quad \text{ \ } \overline{{c}}>\hat{c} >1$\\

\item  If $\alpha\in[-1,0)$ and $n\ge 2$,

$\Psi_{R,T}={C}_\alpha \left(\fr{t+T}{R^2T^\fr{2}{2-n\alpha}-|x|^2}\right)_+^\fr{1}{\alpha},\quad ({C}_\alpha)^\alpha = 2n^{\alpha-1}\left(n-\fr{2}{\alpha}\right).$

\end{enumerate}
\end{definition}
$\Psi$ in Definition \ref{def-Psi} are mostly small modifications of Barrenblat and fundamental soltuions of \eqref{eq-target} and they are subsolutions of \eqref{eq-target}. In order to obtain positive local uniform lower bounds, it is crucial to have similar versions of Corollary \ref{cor-general-1} in other ranges of $\alpha$ and $n$. \\

Thus, we are going to show similar  estimates given in Section \ref{subsec-general-2}. First in the case (2), we have: 
\begin{equation}\label{eq-general-5}\begin{aligned}
	\Psi _t - \D \cdot(\fr{\D \Psi}{n^{1-\alpha}\Psi^\alpha}&)=-\left[2n\fr{(T-t)}{T} \fr{R^2e^{\fr{4nt}{T}}}{|x|^2+R^2e^{\fr{4nt}{T}}}+\fr{1}{\fr{t}{T}+1} \right]\fr{\Psi}{\alpha T} \\  \quad & <-\fr{1}{2{\alpha T}} {\Psi}<0
	\end{aligned}
	\end{equation}
	on $0<t<T$. Moreover, 
	\begin{equation}
	\begin{aligned}
|A_i|&\le c_0\fr{\Psi}{T}\left(|x|^2+R^2e^{\fr{4nt}{T}}\right)^\fr{1}{2}, &&|\partial_tA_i|\le c_0\fr{\Psi}{T^2}\left(|x|^2+R^2e^{\fr{4nt}{T}}\right)^\fr{1}{2},\\
|B_i|&\le c_0\fr{\Psi}{T^2}\left(|x|^2+R^2e^{\fr{4nt}{T}}\right), &&|D_{v_i}B_i|\le c_0\fr{\Psi}{T^2}\left(|x|^2+R^2e^{\fr{4nt}{T}}\right)^\fr{1}{2},\\
|\partial_t B_i| &\le  c_0\fr{\Psi}{T^3}\left(|x|^2+R^2e^{\fr{4nt}{T}}\right),&& |\Psi^\alpha A_i|\le c_0 \Psi\left(|x|^2+R^2e^{\fr{4nt}{T}}\right)^{-\fr{1}{2}},\\
\text{and }&\quad |\Psi^\alpha B_i|\le c_0 \fr{\Psi}{T}  
\end{aligned}
\end{equation} on $0<t<T$ where $c_0=c_0(n,\alpha)$.
\begin{corollary}\label{cor-general2-1}
({\it c.f. Corollary \ref{cor-general-1}}) Suppose  $\alpha=\fr{2}{n}$ and $n\ge2$. For all $\e>0$ and $\Psi_{R,T}$ of Definition \ref{def-Psi}, there exists a family $\{{\overline u_i^\e}\}_{i=1}^{2n}$ which is a subsolution of \eqref{eq-main} satisfying $\fr{1}{4n}\Psi \le \overline{ u_i^\e} \le \fr{3}{4n}\Psi$ on $\left(|x|^2+R^2e^{\fr{4nt}{T}}\right)^\fr{1}{2}<\fr{c}{\e}T$ and $t\in[0,T)$. The constant $c$ depends on $n$ and $\alpha$.
\begin{proof}This is just a reproduction of Proposition \ref{prop-general-1} and \ref{prop-general-2}. Every error and remainder estimate in those propositions could be rewritten by replacing $\fr{(|x|^2+R^2)^\fr{1}{2}}{T-ct}$ with $\fr{(|x|^2+R^2e^{\fr{4nt}{T}})}{T}$.
\end{proof}
\end{corollary}
Next, in the case (3), similar argument as in case (2) gives us the following: 
\begin{equation}\label{eq-general-6}
\begin{aligned}
\Psi _t - \D \cdot(\fr{\D \Psi}{n^{1-\alpha}\Psi^\alpha}&)=-\left[\fr{2}{2-n\alpha}\fr{R^2(t+T)^\fr{4}{2-n\alpha}}{|x|^2+R^2(t+T)^{\fr{4}{2-n\alpha}}}\fr{T-t}{T+t}+1 \right]\fr{\Psi}{T\alpha} \\  \quad & <-\fr{\Psi}{T\alpha}<0
\end{aligned}
\end{equation}
	on $0<t<T$. Moreover,
	\begin{equation}
	\begin{aligned}
|A_i|&\le c_0\fr{\Psi}{T}\left(|x|^2+R^2(t+T)^\fr{4}{2-n\alpha}\right)^\fr{1}{2}, &&|\partial_tA_i|\le c_0\fr{\Psi}{T^2}\left(|x|^2+R^2(t+T)^\fr{4}{2-n\alpha}\right)^\fr{1}{2},\\
|B_i|&\le c_0\fr{\Psi}{T^2}\left(|x|^2+R^2(t+T)^\fr{4}{2-n\alpha}\right), &&|D_{v_i}B_i|\le c_0\fr{\Psi}{T^2}\left(|x|^2+R^2(t+T)^\fr{4}{2-n\alpha}\right)^\fr{1}{2},\\
|\partial_t B_i| &\le  c_0\fr{\Psi}{T^3}\left(|x|^2+R^2(t+T)^\fr{4}{2-n\alpha}\right),&& |\Psi^\alpha A_i|\le c_0 \Psi\left(|x|^2+R^2(t+T)^\fr{4}{2-n\alpha}\right)^{-\fr{1}{2}},\\
\text{and }&\quad |\Psi^\alpha B_i|\le c_0 \fr{\Psi}{T}  &&
\end{aligned}
\end{equation} 
on $0<t<T$ where $c_0=c_0(n,\alpha)$.
\begin{corollary}\label{cor-general2-2}
Suppose  $\alpha\in(0,\fr{2}{n})$ and $n\ge2$. For all $\e>0$ and $\Psi_{R,T}$ of Definition \ref{def-Psi}, there exists a family $\{{\overline u_i^\e}\}_{i=1}^{2n}$ which is a subsolution of \eqref{eq-main} satisfying $\fr{1}{4n}\Psi \le \overline{ u_i^\e} \le \fr{3}{4n}\Psi$ on $\left(|x|^2+R^2(t+T)^{\fr{4}{2-n\alpha}}\right)<\fr{c}{\e}T$ and $t\in[0,T)$. The constant $c$ depends on $n$ and $\alpha$.
\begin{proof}It can be proved same as Corollary \ref{cor-general2-1}.
\end{proof}
\end{corollary}

In the case (4), we have the same kinds but slightly different estimates:
\begin{equation}\label{eq-general-7}
\begin{aligned}
\Psi _t - \D \cdot(\fr{\D \Psi}{n^{1-\alpha}\Psi^\alpha}&)=-\left[\fr{\hat{c}(\hat{c}-1)}{2}\fr{|x|^2}{(t+T)^2} + (\overline{c}-\hat{c}) \fr{1}{t+T}\right]\fr{n\Psi}{2} \\  \quad & <-c_1\left[ \fr{1}{t+T}+\fr{|x|^2}{(t+T)^2}\right]\Psi<0
\end{aligned}
\end{equation}
	on $0<t<\infty$ and $c_1=c_1(c,\overline{c})=c_1(n,\alpha)$ when $\hat c$ and $\overline{c}$ are fixed. Moreover, we have
\begin{equation}\label{eq-general-8}
	\begin{aligned}
|A_i|&\le c_0\Psi\left( \fr{|x|}{t+T} \right), &&\\
|\partial_tA_i|&\le c_0\Psi\left(\fr{|x|}{(t+T)^2}+ \fr{|x|^3}{(t+T)^3}\right) = c_0 \left[ \fr{1}{t+T}+\fr{|x|^2}{(t+T)^2}\right]\Psi \fr{|x|}{t+T},\\
|B_i|&\le c_0\Psi\left(\fr{1}{t+T}+ \fr{|x|^2}{(t+T)^2}\right), \\
|D_{v_i}B_i|&\le c_0\Psi\left(\fr{|x|}{(t+T)^2}+ \fr{|x|^3}{(t+T)^3}\right)= c_0 \left[ \fr{1}{t+T}+\fr{|x|^2}{(t+T)^2}\right]\Psi \fr{|x|}{t+T},\\
\text{and } &\quad  |\partial_t B_i| \le  c_0\Psi\left(\fr{1}{(t+T)^2}+ \fr{|x|^2}{(t+T)^3}+\fr{|x|^4}{(t+T)^4}\right)
\end{aligned}
\end{equation} on $0<t<\infty$ where $c_0=c_0(n,\alpha)$.

Therefore following similar argument of Proposition \ref{prop-general-1} and \ref{prop-general-2}, we obtain a version of Corollary \ref{cor-general-1} in the case (4). 

\begin{corollary}\label{cor-general2-3}
Suppose  $\alpha =0$ and $n\ge2$. There exists $\e_0(n,\alpha,T)$ such that for all $0<\e<\e_0(n,\alpha,T)>0$ and $\Psi_{R,T}$ of Definition \ref{def-Psi}, there exists a family $\{{\overline u_i^\e}\}_{i=1}^{2n}$ which is a subsolution of \eqref{eq-main} satisfying $\fr{1}{4n}\Psi \le \overline{ u_i^\e} \le \fr{3}{4n}\Psi$ on ${|x|}<\fr{c}{\e}(t+T)$ and $t\in[0,\infty)$. the constant $c$ depends on $n$ and $\alpha$.
\begin{proof}
It comes from the same line of the argument as Corollaries above, but we need an extra dependence of $\e$ with respect to $T$ since we need $\fr{c_0}{t+T}$ and $\fr{c_0}{(t+T)^2}$ to be small in the proof. The proof is much simpler except for slightly different estimates \eqref{eq-general-8}. Actually when $\alpha=0$ in \eqref{eq-expansion}, the remainder term $R$ and the other terms having $\alpha$ or $\binom \alpha2$ coefficients are zero. 
\end{proof}
\end{corollary}

Finally in the case (5), since those expressions we are going to use involve sign operator, the derivatives are not classical, but they are given in weak notion.  $\Psi$ doesn't have weak second derivatives because second derivatives blow up at the boundary of the support. However, $\Psi^{1-\alpha}$ has second derivatives. In fact, we can check that 
\begin{equation}\begin{aligned}\partial_i \Psi^{1-\alpha} &= (C_\alpha)^{1-\alpha} \partial_i \left| \left(\fr{R^2T^\fr{2}{2-n\alpha}-|x|^2}{t}\right)\right|^\fr{1-\alpha}{-\alpha} _+\\&=(C_\alpha)^{1-\alpha} \fr{1-\alpha}{-\alpha} \left(\fr{R^2T^\fr{2}{2-n\alpha}-|x|^2}{t}\right)^{-\fr{1}{\alpha}} sgn^+\left(\fr{R^2T^\fr{2}{2-n\alpha}-|x|^2}{t}\right)2x_i\\ &=\fr{(1-\alpha)(C_\alpha)^{1-\alpha}x_i}{-\alpha}\left| \left(\fr{R^2T^\fr{2}{2-n\alpha}-|x|^2}{t}\right)\right|^{-\fr{1}{\alpha}} _+\end{aligned}\end{equation} and the last expression has first order weak derivatives. Thus, expressions like $\fr{\D \Psi}{\Psi^\alpha}$ need to be interpreted as $\D \fr{\Psi^{1-\alpha}}{1-\alpha}$.

From this we can do actual computations and get the following estimate: 
\begin{equation}\label{eq-general-9}
\begin{aligned}
\Psi _t - \D \cdot\left(\fr{\D \Psi}{n^{1-\alpha}\Psi^\alpha}\right)&=\fr{1}{\alpha}\fr{2}{2-n\alpha}\left[\fr{R^2T^\fr{2}{2-n\alpha} }{R^2T^\fr{2}{2-n\alpha}-|x|^2}\right]_+\fr{\Psi}{t+T} \\  \quad & \le\fr{1}{\alpha}\fr{1}{2-n\alpha}\left[\fr{R^2T^\fr{2}{2-n\alpha}+|x|^2 }{R^2T^\fr{2}{2-n\alpha}-|x|^2}\right]_+\fr{\Psi}{t+T}.
\end{aligned}
\end{equation}

In addition, we can check	\begin{equation}
	\begin{aligned}
|A_i|&\le c_0\fr{\Psi}{(t+T)}\left(|x|^2+R^2T^\fr{2}{2-n\alpha}\right)^\fr{1}{2}, &&|\partial_tA_i|\le c_0\fr{\Psi}{(t+T)^2}\left(|x|^2+R^2T^\fr{2}{2-n\alpha}\right)^\fr{1}{2}, \\
|B_i|&\le c_0\fr{\Psi}{(t+T)^2}\left(|x|^2+R^2T^\fr{2}{2-n\alpha}\right),  &&|D_{v_i}B_i|\le c_0\fr{\Psi}{(t+T)^2}\fr{\left(|x|^2+R^2T^\fr{2}{2-n\alpha}\right)^\fr{3}{2}}{\left(R^2T^\fr{2}{2-n\alpha}-|x|^2\right)_+},\\
|\partial_t B_i| &\le c_0\fr{\Psi}{(t+T)^3}\left(|x|^2+R^2T^\fr{2}{2-n\alpha}\right)\le&& c_0\fr{\Psi}{(t+T)^3}\fr{\left(|x|^2+R^2T^\fr{2}{2-n\alpha}\right)^2}{\left(R^2T^\fr{2}{2-n\alpha}-|x|^2\right)_+},\\ |\Psi^\alpha A_i|&\le c_0 \Psi\fr{\left(|x|^2+R^2T^\fr{2}{2-n\alpha}\right)^\fr{1}{2}}{\left(R^2T^\fr{2}{2-n\alpha}-|x|^2\right)_+},\text{ \quad and }&&|\Psi^\alpha B_i|\le c_0 \fr{\Psi}{t+T}\fr{\left(|x|^2+R^2T^\fr{2}{2-n\alpha}\right)}{\left(R^2T^\fr{2}{2-n\alpha}-|x|^2\right)_+}
\end{aligned}
\end{equation} 
for some $c_0=c_0(n,\alpha)$. \\

Now we can formulate our last corollary in this case (5). 
\begin{corollary}\label{cor-general2-4}
Suppose  $\alpha \in [-1,0)$ and $n\ge2$. Then, there exist $\e_0(n,\alpha)>0$ such that for all $0<\e<\e_0(n,\alpha)$ and $\Psi_{R,T}$ of Definition \ref{def-Psi}, there exists a family $\{{\overline u_i^\e}\}_{i=1}^{2n}$ which is a subsolution of \eqref{eq-main} satisfying $\fr{1}{4n}\Psi \le \overline{ u_i^\e} \le \fr{3}{4n}\Psi$ on $\left(|x|^2+R^2T^\fr{2}{2-n\alpha}\right)^\fr{1}{2}<\fr{c}{\e}(t+T)$ and $t\in[0,\infty)$. The constant $c$ depends on $n$ and $\alpha$.

\begin{proof}
It can be proved same as corollaries above.\end{proof}
\end{corollary}

Now we are ready to prove the proposition for positive lower bounds.
\begin{prop}\label{prop-general2-1}
For $n\ge2$ and $|\alpha|\le 1$, suppose mild solution, $u_i^\e$, of \eqref{eq-main}    has initial data $g_i \in L_{loc}^1(\R^n)$ with lower bound $g_i\ge \fr{3}{4n}\Psi_{R,T}$ for some $R$ and $T>0$. Then there exists universal constants $C=C(n,\alpha)>0$ such that for every space-time compact set $K\subset \R^n\times[0,CT)$, there exist $\e_K>0$ satisfying $u_i^\e \ge \fr{1}{4n}\Psi_{R,T}$ on K for $0<\e<\e_K$.

\begin{proof}
We first consider case $(1)$ of Definition \ref{def-Psi}.
Let us choose $C=\fr{1}{2} min\left(\overline{c},\fr{1}{c}\right)$ using constants in Corollary \ref{cor-general-1}. Then, the constructed explicit local subsolution $\overline{u_i^\e}$ in Proposition \ref{prop-general-2} is a subsolution of \eqref{eq-main} satisfying $\fr{1}{4n}{\Psi} \le \overline{u_i^\e} \le \fr{3}{4n}{\Psi}$ on \begin{equation}\label{eq-domain}K_1^\e=\{(x,t)\in Q\ |\ \e(|x|^2+R^2)^\fr{1}{2}\le CT\text{ and }t\in[0,CT]\}.\end{equation}For any $K=Q_{r,t_0}$ with $0<t_0<CT$, there exists $\e_K>0$ such that$$\fr{t_0}{\e}+r < \left(\left(\fr{CT}{\e}\right)^2-R^2\right)^\fr{1}{2}\text{ if }0<\e < \e_K.$$
This implies $Q_{r+\fr{t_0}{\e},t_0} \subset K_1^\e$, and now we can apply comparison principle of Theorem \ref{thm-MP} $(ii)$ between $\overline{u_i^\e}$ and $u_i^\e$. Then we have $u_i^\e\ge \fr{1}{4n} \Psi$ on $K=Q_{r,t_0}$.\\

The other cases of Definition \ref{def-Psi} could be proved in similar ways. As \eqref{eq-domain}, we define 

\begin{equation}
K_1^\e =\begin{cases} \{(x,t)\in Q\ |\ \left(|x|^2+R^2e^{\fr{4nt}{T}}\right)^\fr{1}{2}<\fr{c}{\e}T\text{ and }t\in[0,T])   & \text{in case } (2)\\   
\{(x,t)\in Q\ |\ \left(|x|^2+R^2(t+T)^{\fr{4}{2-n\alpha}}\right)<\fr{c}{\e}T\text{ and }t\in[0,T]) & \text{in case } (3)\\
\{(x,t)\in Q\ |\ |x|<\fr{c}{\e}(t+T)\text{ and }t\in[0,\infty)) & \text{in case } (4)\\ 
\{(x,t)\in Q\ |\ \left(|x|^2+R^2T^\fr{2}{2-n\alpha}\right)^\fr{1}{2}<\fr{c}{\e}(t+T)\text{ and }t\in[0,\infty)) & \text{in case } (6)\end{cases}
\end{equation}
where each constant $c$ comes from Corollary \ref{cor-general2-1}, \ref{cor-general2-2}, \ref{cor-general2-3}, and \ref{cor-general2-4} respectively.
Let us choose $C=min(c,1)>0$, and then it is easy to check that for any $K=Q_{r,t_0}$ with $0<t_0<CT$ there exist samll $e_K$ such that $Q_{r+\fr{t_0}{\e},t_0} \subset K_1^\e$ for $\e<\e_K$. Applying comparison principle of Theorem \ref{thm-MP} $(ii)$ again, we prove the proposition. 

\end{proof}
\end{prop}

Applying Proposition \ref{prop-general2-1}, we can prove the convergence in diffusive limit under certain general conditions on initial data. As mentioned before, we can divide six cases into two groups depending on whether Barenblatt type solution has finite mass or not. 

\begin{prop}\label{prop-general2-2}
For $n\ge2$ and $|\alpha|\le1$, suppose initial data $g_i$ in \eqref{eq-main} are locally integrable, which mean $g_i\in L^1_{loc}(\R^n)$.  
\begin{enumerate}[(i)]
\item For $\alpha\in[\fr{2}{n},1]$, {\it i.e.} case (1) and (2), if there exist $f_i\in L^\infty$ $i=1,2,\ldots,2n$ satisfying $g_i-f_i\in L^1(\R^n)$ and $f_i\ge \fr{3}{4n}\Psi_{R,T}(0)$ for some $R>0$ and $T>0$, then as $\e\to0$ we have$$\text{ $u_i^\e \rightarrow \fr{\rho}{2n}$ in $L^1_{loc}(\R^n\times(0,CT))$}$$  where $\rho$ is a weak solution of \eqref{eq-fastdiff} with initial data $\sum_i g_i$. Moreover, under the condition \eqref{eq-compactcondition}, we have$$\text{$\rho^\e \rightarrow \rho$ in $C([0,CT),L^1_{loc}(\R^n))$ .}$$ $C=C(n,\alpha)$ is constant obtained in Proposition \ref{prop-general2-1}. 

For $\alpha=1$, if \eqref{eq-fastdiff} has no uniqueness of solution for the initial data, the convergence takes along a subsequence for each given sequence $\e_j\to0$. Otherwise, they convergence arbitrarily as $\e \to0$.
\item For $\alpha\in[-1,\fr{2}{n})$, {\it i.e.} case (4),(5), and (6), if there exist $f_i\in L^\infty$ $i=1,2,\ldots,2n$ satisfying $g_i-f_i\in L^1(\R^n)$, then as $\e\to0$ we have $$\text{ $u_i^\e \rightarrow \fr{\rho}{2n}$ in $L^1_{loc}(\R^n\times(0,\infty))$}$$ where $\rho$ is the unique weak solution of \eqref{eq-fastdiff} with initial data $\sum_i g_i$. Moreover, under the condition \eqref{eq-compactcondition}, we have $$\text{$\rho^\e \rightarrow \rho$ in $C([0,T],L^1_{loc}(\R^n))$ for all $T>0$}$$ where the limit is unique. 

\end{enumerate}

\begin{proof}
For any $n$ and $\alpha$, suppose $g_i\ge \fr{3}{4n}\Psi_{\alpha,n,R,T} (0)$ for some $R$ and $T>0$ and uniformly bounded, say $g_i\le M$. By Proposition \ref{prop-general2-1}, we have locally uniform positive lower bound for small $\e>0$ on $t\in[0,CT)$. As it is pointed, positive local lower and upper bounds are sufficient to prove diffusive limit of \eqref{eq-main}. Thus, we can simply adopt Theorem \ref{thm-dl-1} and  obtain $u_i^\e \rightarrow \fr{\rho}{2n}$ in $L^2_{loc}(\R^n\times(0,CT))$ and $\rho^\e \rightarrow \rho$ in $C([0,CT),L^1_{loc}(\R^n))$ convergence. 

Now in the case $(i)$, $$[\fr{3}{4n}\Psi_{nR,T}(0)-g_i]_+\le [\fr{3}{4n}\Psi_{R,T}(0)-g_i]_+ \le [f_i-g_i]_+ \in L^1$$ and  $$[g_i-(f_i+n)]_+\le[g_i-f_i]_+\in L^1.$$ Thus if we define $$g_{i,n}:=\begin{cases}\begin{aligned} &f_i+n &&\text{if }g_i\ge f_i+n \\&g_i &&\text{if }f_i+n > g_i\ge\fr{3}{4n}\Psi_{nR,T} \\&\fr{3}{4n}\Psi_{nR,T} &&otherwiese,\end{aligned}\end{cases}$$ we have  $g_{i,n}- g_i\to0$ in $L^1$ as $n\to \infty$ by dominated convergence theorem. Suppose $u^\e_{i,n}$ are solutions of \eqref{eq-main} with initial data $g_{i,n}$. Set $\rho_n$ and  $\rho$ to be solutions of \eqref{eq-fastdiff} with initial data $\sum g_{i,n}$ and $\sum g_i$, respectively. From the result above, we have $u_{i,n}^\e \rightarrow \fr{\rho_n}{2n}$ in $L^2_{loc}(\R^n\times(0,CT))$ and $\sum_i u^\e_{i,n} \rightarrow \rho_n$ in $C([0,CT),L^1_{loc}(\R^n))$. Moreover, $L^1$-contraction among solutions of \eqref{eq-main} and solutions of \eqref{eq-fastdiff} says $$\sum_i  \left\Vert u^\e_{i,n}(t)-u^\e_i(t)\right\Vert_1 \le \sum_i \left\Vert g_{i,n}-g_i\right\Vert_1 \to 0 $$ and $$\left\Vert \rho_n(t)-\rho(t)\right\Vert_1\le \big\Vert \sum_i(g_{i,n}-g_i)\big\Vert_1 \to 0$$ as $n\to\infty$. Observations above lead us to prove $(i)$. We can only prove a convergence along some subsequence to some solution $\rho$ of \eqref{eq-fastdiff} in case $\alpha=1$ since there is no uniqueness of solution in that case with this decay assumption. 

Case $(ii)$ is not much different from $(i)$. For any given $T>0$, let us define $$g_{i,n}:=\begin{cases}\begin{aligned}&f_i+n &&\text{if }g_i\ge f_i+n \\&g_i &&\text{if }f_i+n > g_i\ge\fr{3}{4n}\Psi_{n,T} \\&\fr{3}{4n}\Psi_{n,T} &&otherwiese,\end{aligned}\end{cases}$$  then $g_{i,n} - g_i \to 0$ in $L^1$ as $n\to \infty$ as before. By the same line of argument, we get the result.
\end{proof}
\end{prop}

\begin{remark}
In the case $(1)$ and $(2)$, if $f_i$ has thicker lower bound $\liminf_{|x|\to\infty}|x|^\fr{2}{\alpha}f_i = \infty$, then we can prove the convergence locally upto infinite time. {\it i.e.} $$\text{ $u_i^\e \rightarrow \fr{\rho}{2n}$ in $L^1_{loc}(\R^n\times(0,\infty))$ and }$$ $$\text{$\rho^\e \rightarrow \rho$ in $C([0,T),L^1_{loc}(\R^n))$ for all $T>0$ under the condition \eqref{eq-compactcondition}}.$$
\end{remark}

So far, we have forcused on obtaining positive lower bounds, and thus our $g_i$ was chosen to be $L^1$-perturbation of $L^\infty$ function $f_i$. However, through the same barrier argument, we may get local upper bound in the same way as we got the lower bound in Proposition \ref{prop-general2-1}. We need to construct supersolutions for this case.   As $\rho$ gets larger, diffusion gets faster in PME range. Therefore, the finite time blow-up arises in PME range with fast growing initial data as the finite time extinction happens in FDE range with fast decaying initial data.\\

   Now we start with the case (5), which gives us PME in the limit. For $\alpha\in[-1,0)$ and $n\ge 2$, we may define $$\overline{\Psi}_{R,T}(x,t)=C_\alpha \left(\fr{T-ct}{|x|^2+R^2}\right)^\fr{1}{\alpha}, \ (C_\alpha)^\alpha=2n^{\alpha-1}\left(n-\fr{2}{\alpha}\right)\quad c>1.$$Then we have: 
\begin{equation}\label{eq-general-10}\begin{aligned}
	\overline{\Psi} _t - \D \cdot(\fr{\D \overline{\Psi}}{n^{1-\alpha}\overline{\Psi}^\alpha}&)=\left[c+\fr{n\alpha}{2-n\alpha}-\fr{2}{2-n\alpha}\fr{|x|^2}{|x|^2+R^2}\right]\fr{-\overline{\Psi}}{\alpha( T-ct)} \\  \quad & >(c-1)\fr{-\overline{\Psi}}{\alpha( T-ct)}>0
	\end{aligned}
	\end{equation}
	on $0<ct<T$. Moreover, 
	\begin{equation}
	\begin{aligned}
|A_i|&\le c_0\fr{\overline{\Psi}}{T-ct}\left(|x|^2+R^2\right)^\fr{1}{2}, &&|\partial_tA_i|\le c_0\fr{\overline{\Psi}}{T-ct^2}\left(|x|^2+R^2\right)^\fr{1}{2},\\
|B_i|&\le c_0\fr{\overline{\Psi}}{T-ct^2}\left(|x|^2+R^2\right), &&|D_{v_i}B_i|\le c_0\fr{\overline{\Psi}}{T-ct^2}\left(|x|^2+R^2\right)^\fr{1}{2},\\
|\partial_t B_i| &\le  c_0\fr{\overline{\Psi}}{T-ct^3}\left(|x|^2+R^2\right),&& |\overline{\Psi}^\alpha A_i|\le c_0 \overline{\Psi}\left(|x|^2+R^2\right)^{-\fr{1}{2}},\\
\text{and }&\quad |\overline{\Psi}^\alpha B_i|\le c_0 \fr{\overline{\Psi}}{T-ct}  
\end{aligned}
\end{equation} on $0<ct<T$ where $c_0=c_0(n,\alpha)$.
\begin{corollary}\label{cor-general2-5}
({\it c.f. Corollary \ref{cor-general-1}}) Suppose $\alpha\in[-1,0)$ and $n\ge 2$. For all $\e>0$ and $\overline{\Psi}_{R,T}$ of Definition above, there exists a family $\{{\overline u_i^\e}\}_{i=1}^{2n}$ which is a supersolution of \eqref{eq-main} satisfying $\fr{1}{4n}\overline{\Psi} \le \overline{ u_i^\e} \le \fr{3}{4n}\overline{\Psi}$ on $(|x|^2+R^2)^\fr{1}{2}<\fr{c}{\e}(T-ct)$ and $t\in[0,\fr{T}{c})$. The constant $c>1$ depends on $n$ and $\alpha$.
\begin{proof}Proof is the same as previous corollaries.
\end{proof}
\end{corollary}

In the case (4), $\alpha=0$, we define $$\overline{\Psi}_{R,T}(x,t)=\fr{R^2}{\left(4\pi {(T-t)}\right)^{\fr{n}{2}\overline{c}}}e^{\fr{\hat{c}n|x|^2}{4(T-t)}},\quad 0<\hat{c}<1\text{ and } \overline{c}>\hat{c}$$
\begin{equation}\label{eq-general2-14}
\begin{aligned}
\overline{\Psi} _t - \D\cdot\left(\fr{\D \overline{\Psi}}{n^{1-\alpha}\overline{\Psi}^\alpha}\right)&=\left[\fr{\hat{c}(1-\hat{c})}{2}\fr{|x|^2}{(T-t)^2} + (\overline{c}-\hat{c}) \fr{1}{T-t}\right]\fr{\overline{\Psi}}{2} \\  \quad & >c_1\left[ \fr{1}{T-t}+\fr{|x|^2}{(T-t)^2}\right]\overline{\Psi}>0
\end{aligned}
\end{equation}
	on $0<t<\infty$ and $c_1=c_1(c,\overline{c})=c_1(n,\alpha)$ when $\hat{c}$ and $\overline{c}$ are fixed. Moreover, we have
\begin{equation}\label{eq-general2-15}
	\begin{aligned}
|A_i|&\le c_0\overline{\Psi}\left( \fr{|x|}{T-t} \right), &&\\
|\partial_tA_i|&\le c_0\overline{\Psi}\left(\fr{|x|}{(T-t)^2}+ \fr{|x|^3}{(T-t)^3}\right) = c_0 \left[ \fr{1}{T-t}+\fr{|x|^2}{(T-t)^2}\right]\overline{\Psi} \fr{|x|}{T-t},\\
|B_i|&\le c_0\overline{\Psi}\left(\fr{1}{T-t}+ \fr{|x|^2}{(T-t)^2}\right), \\
|D_{v_i}B_i|&\le c_0\overline{\Psi}\left(\fr{|x|}{(T-t)^2}+ \fr{|x|^3}{(T-t)^3}\right)= c_0 \left[ \fr{1}{T-t}+\fr{|x|^2}{(T-t)^2}\right]\overline{\Psi} \fr{|x|}{T-t},\\
\text{and } &\quad  |\partial_t B_i| \le  c_0\overline{\Psi}\left(\fr{1}{(T-t)^2}+ \fr{|x|^2}{(T-t)^3}+\fr{|x|^4}{(T-t)^4}\right)
\end{aligned}
\end{equation} on $0<t<\infty$ where $c_0=c_0(n,\alpha)$.

Following the similar argument of Proposition \ref{prop-general-1} and \ref{prop-general-2}, therefore, we obtain a version of Corollary \ref{cor-general-1} in the case (4). 

\begin{corollary}
Suppose  $\alpha =0$ and $n\ge2$. There exist $\e(n,\alpha,T)$ such that for all $0<\e<\e(n,\alpha,T)$ and $\overline{\Psi}_{R,T}$ of Definition above, there exists a family $\{{\overline u_i^\e}\}_{i=1}^{2n}$ which is a supersolution of \eqref{eq-main} satisfying $\fr{1}{4n}\overline{\Psi} \le \overline{ u_i^\e} \le \fr{3}{4n}\overline{\Psi}$ on ${|x|}<\fr{c}{\e}(T-t)$ and $t\in[0,T)$. The constant $c$ depends on $n$ and $\alpha$.
\begin{proof}Same as Corollary \ref{cor-general2-3}.
\end{proof}
\end{corollary}

Now in remaining cases (1),(2) and (3), since $$\overline{\Psi}_t-\D\cdot \left(\fr{\overline{\Psi}}{n^{1-\alpha}\overline{\Psi}^\alpha}\right) = \overline{\Psi}_t - \left(\fr{n}{\overline{\Psi}}\right)^\alpha \fr{\Delta \overline{\Psi}}{n} + \alpha\left(\fr{n}{\overline{\Psi}}\right)^\alpha\fr{|\D\overline{\Psi}|^2}{n\overline{\Psi}}.$$
It could be observed that  $$\overline{\Psi}_{0,n,R,T}(x,t)=\fr{R^2}{\left(4\pi {(T-t)}\right)^{\fr{n}{2}\overline{c}}}e^{\fr{\hat{c}n|x|^2}{4(T-t)}},\quad 0<\hat{c}<1\text{ and } \overline{c}>\hat{c}$$ is again a supersolution of \eqref{eq-fastdiff} even when $\alpha\in(0,1]$ for  $R>R(\alpha,n,T)>0$. Moreover when $\alpha\in(0,1]$, the same estimates \eqref{eq-general2-14} and \eqref{eq-general2-15} with appropriate $c_0$ and $c_1$ could be obtained if $R$ is large. 

\begin{corollary}
Suppose  $\alpha\in (0,1]$ and $n\ge2$. There exist $\e_0(n,\alpha,T)>0$ and $R(n,\alpha,T)>0$ such that for $\overline{\Psi}_{R,T}$ of \eqref{eq-general2-14} with $0<\e<\e_0(n,\alpha,T)$ and $R>R(n,\alpha,T)>0$  , there exists a family $\{{\overline u_i^\e}\}_{i=1}^{2n}$ which is a supersolution of \eqref{eq-main} satisfying $\fr{1}{4n}\overline{\Psi} \le \overline{ u_i^\e} \le \fr{3}{4n}\overline{\Psi}$ on ${|x|}<\fr{c}{\e}(T-t)$ and $t\in[0,T)$. Constant $c$ depends on $n$ and $\alpha$.
\begin{proof}Same as Corollary \ref{cor-general2-5}.
\end{proof}
\end{corollary}

It is not hard to show a proposition which corresponds to Proposition \ref{prop-general2-1} in the upper bound case. 
\begin{prop}\label{prop-general2-3}
For $n\ge2$ and $|\alpha|\le 1$, suppose mild solution, $u_i^\e$, of \eqref{eq-main}    has initial data $g_i \in L_{loc}^1(\R^n)$ with upper bound $g_i\le \fr{1}{4n}\overline{\Psi}_{R,T}$ for some $R$ and $T>0$. Then there exists universal constants $\overline{C}=\overline{C}(n,\alpha)>0$ such that for every space-time compact set $K\subset \R^n\times[0,\overline{C}T)$, there exist $\e_K>0$ satisfying $u_i^\e \le \fr{3}{4n}\overline{\Psi}_{R,T}$ on K for $0<\e<\e_K$.
\end{prop}

We can formulate our main result, Theorem \ref{thm-mmainthm}.

\begin{subsection} {Proof of Theorem \ref{thm-mmainthm}}  

The proof would be very similar to Proposition \ref{prop-general2-2}. First, from the definition of $T_1$ in Theorem \ref{thm-mmainthm}, there exist $C_1(n,\alpha)>0$ such that $[\fr{3}{4n}\Psi_{1,T}-g_i]_+\in L^1(\R^n)$  for all $0<T<C_1T_1$. Similarily, there exist $C_2(n,\alpha)>0$ such that $[g_i-\fr{1}{4n}\overline{\Psi}_{1,T}]_+\in L^1(\R^n)$ for all $T<C_2T_{2}$. Thus for $0<T<min(C_1T_1,C_2T_2)$, if we define $$g_{i,m}:=\begin{cases}\begin{aligned} &\fr{1}{4n}\overline{\Psi}_{m,T}&&\text{if }\fr{1}{4n}\overline{\Psi}_{m,T}\le g_i\\&g_i &&\text{if }\fr{3}{4n}\Psi_{m,T}< g_i\le\fr{1}{4n}\overline{\Psi}_{m,T}
 \\&\fr{3}{4n}\Psi_{m,T}   &&otherwiese,\end{aligned}\end{cases}$$ then we have $g_{i,m}\to g_i$ in $L^1(\R^n)$ as $n\to \infty$. 

We are going to use notations $u^\e_{i,n}$ and $\rho_n$ defined in Proposition \ref{prop-general2-2}. For each fixed $m\in\N$ and $0<T<min(CC_1, \overline{C}C_2)\cdot min(T_1,T_2)$, if $K$ is given compact set in $\R^n\times [0,T]$, there exist $\e_K>0$ such that $\fr{1}{4n}\Psi_{m,T}\le u^\e_{i,m} \le \fr{3}{4n}\overline{\Psi}_{m,T}$ on $K$ for $0<\e<\e_K$ by Proposition \ref{prop-general2-1} and \ref{prop-general2-3}. Constants $C$ and $\overline{C}$ came from the two propositions. Using these local upper and lower bounds, it is clear that $u_{i,m}^\e \rightarrow \fr{\rho_m}{2n}$ in $L^2_{loc}(\R^n\times(0,T))$ and $\rho^\e_m \rightarrow \rho_m$ in $C([0,T],L^1_{loc}(\R^n))$ convergence under the condition \eqref{eq-compactcondition}.

Finally, $L^1$-contraction among solutions of \eqref{eq-main} and solutions of \eqref{eq-fastdiff} says $$\sum_i  \left\Vert u^\e_{i,m}(t)-u^\e_i(t)\right\Vert_1 \le \sum_i \left\Vert g_{i,m}-g_i\right\Vert_1 \to 0\text{, \quad and} $$ $$\left\Vert \rho_m(t)-\rho(t)\right\Vert_1\le \big\Vert \sum_i(g_{i,m}-g_i)\big\Vert_1 \to 0$$ as $m\to\infty$ and this proves the theorem.

\begin{remark}
In order to construct barriers, we have assumed that the initial data have lower bounds $O\left({|x|^{-\fr{2}{\alpha}}}\right)$ for $\alpha\ge \fr{2}{n}$ or upper bounds  $O\left({|x|^{-\fr{2}{\alpha}}}\right)$ for $\alpha<0$ to obtain estimates corresponding to \eqref{eq-general-4-1} and \eqref{eq-general-4-2}. Those decay and growth rates are kinds  of optimal rates where our method works. 

On the other hand, those rates are kinds of crucial rartes in the theory of FDE and PME. When $\alpha>\fr{2}{n}$, initial data with this decay profiles are typical members of Marcinkiewicz space $M^{\fr{n\alpha}{2}}(\R^n)$. This is considered as natural exitinction space of FDE. {\it i.e.} it is a critical rate where extiction behavior take place. One can refer \cite{V2} for more detailed theory about extinction behavior and role of  Marcinkiewicz space in that theory. When $\alpha<0$, growth rate $|x|^{-\fr{2}{\alpha}}$ is critical rate in blow-up phenomeonon and existence theory. It is known that \eqref{eq-fastdiff} with initial data satisfying $\int_{B_R}\rho_0 dx \sim O(R^{n-\fr{2}{\alpha}})$ has an existence of solution in a short time and $o(R^{n-\fr{2}{\alpha}})$ has a long time existence. Notice that our growth assumption in this paper is slightly restrictive than this integral growth assumption above. More detail could be found at standard texts such as \cite{DK} and \cite{V}.
\end{remark}

\end{subsection}

\begin{subsection}{Notes on $n=1$ and other interaction rates}

When $n=1$ and $|\alpha|\le 1$, our equation is exactly same as the equation investigated in \cite{SV}. For the case $\alpha>1$, the system has no $L^1$-contraction lemma. Even in this case, we have comparison principle if solutions are locally bounded, \cite{SV} \cite{NH}. For $1<\alpha<2$, there is a family of global in time, source type Barenblatt solution of target ultra fast diffusion equation. Using this and the same method we used, we may prove diffusive limit convergence if initial data has a lower bound of Barenblatt type profile. For $\alpha\ge 2$, target equation has a Barenblatt type solution which is not in $L^1$. This corresponse to $\fr{2}{n}\le \alpha\le 1$ cases of higher dimension. Thus, we may prove convergence for a short time, whose length depending on tail thickness. Unlike other higher dimensional $L^1$ contractive case ,  we may not weaken the initial condition as $L^1$ purtabation of the Barenblatt profile. We can also deal with growing unbounded initial data in the same way. \\

On the other hand, there are other interaction terms which give us the same target FDE or PME for $n\ge1$. First of all, we may use interaction term proposed in \cite{TL} in their higher dimension model: 
\begin{equation}A(x)\left(\begin{array}{c}u_1\\u_2\\\vdots\\u_{2n}\end{array}\right)=
\left(\begin{array}{c}\sum_{j=1}^{2n}\rho^\alpha(u_j-u_1)\\ \sum_{j=1}^{2n}\rho^\alpha(u_j-u_2)\\ \vdots \\\sum_{j=1}^{2n}\rho^\alpha(u_j-u_{2n}) \end{array}\right)=
\left(\begin{array}{c}\rho^\alpha(\rho-2nu_1)\\ \rho^\alpha(\rho-2nu_2)\\ \vdots \\\rho^\alpha(\rho-2nu_{2n}) \end{array}\right).
\end{equation}
By taking derivative of $f(x,y)=(x+y)^\alpha(x-(2n-1)y)$ w.r.t. $x, y>0$, we can prove above system has $L^1$-contraction for $-1\le\alpha\le \fr{1}{2n-1}$ and comparison priciple for other $\alpha$ among locally bounded solutions. Next, \begin{equation}A(x)\left(\begin{array}{c}u_1\\u_2\\\vdots\\u_{2n}\end{array}\right)=
\left(\begin{array}{c}u_j^{(\alpha+1)}-u_1^{(\alpha+1)}\\ u_j^{(\alpha+1)}-u_2^{(\alpha+1)}\\ \vdots \\u_j^{(\alpha+1)}-u_{2n}^{(\alpha+1)} \end{array}\right)\end{equation} gives $L^1$-contraction for $\alpha \ge -1$ and comparison principle for other $\alpha$. The same method could be employed to prove limit convergence and limit equation will be the same FDE and PME $\rho_t=\Delta \rho^{1-\alpha}$ which are the same upto constant factors. 

\end{subsection}

\section*{Acknowledgement}
We would like to thank L.C. Evans for his encouragement to this project.
Beomjun Choi has been supported by National Institute for Mathematical Sciences (NIMS) grants funded by the
Korea government (No. A21502). Ki-Ahm Lee  was supported by the National Research Foundation of Korea(NRF) grant funded by the Korea government(MSIP) \\ (No.2014R1A2A2A01004618).  Ki-Ahm Lee also hold a joint appointment with the Research Institute of Mathematics of Seoul National University.


\begin{thebibliography}{99}
\bibitem[C] {C}
T. Carleman. {\it Sur la th\'{e}orie de l'\'{e}quation int\'{e}grodifférentielle de Boltzmann.} Acta Mathematica 60.1 (1933): 91-146.

\bibitem[C2] {C2}
T. Carleman. {\it Problemes math\'{e}matiques dans la th\'{e}orie cin\'{e}tique de gaz.} Vol. 2. Almqvist \&\ Wiksell, 1957.


\bibitem[K] {K}
T. G. Kurtz. {\it Convergence of sequences of semigroups of nonlinear operators with an application to gas kinetics.} Transactions of the American Mathematical Society 186 (1973): 259-272.

\bibitem[M] {M}
F. Murat. {\it Compacit\'{e} par compensation.} Annali della Scuola Normale Superiore di Pisa-Classe di Scienze 5.3 (1978): 489-507.

\bibitem[NH] {NH}
R. Natalini, and B. Hanouzet. {\it Weakly coupled systems of quasilinear hyperbolic equations.} Differential and Integral Equations 9.6 (1996): 1279-1292.

\bibitem[S]{S}
J. Simon. {\it Compact sets in the spaceL p (O, T; B).} Annali di Matematica pura ed applicata 146.1 (1986): 65-96.

\bibitem[McK] {McK}
H. P. McKean. {\it The central limit theorem for Carleman’s equation.} Israel Journal of Mathematics 21.1 (1975): 54-92.

\bibitem[SR] {SR}
L. Saint-Raymond {\it Hydrodynamic limits of the Boltzmann equation.} No. 1971. New York: Springer, 2009.


\bibitem[HKS] {HKS}
H. J. Hwang, K. Kang, and A. Stevens. {\it Drift-diffusion limits of kinetic models for chemotaxis: a generalization.} Discrete Contin. Dyn. Syst. Ser. B 5.2 (2005): 319-334.

\bibitem[CAB] {CAB}
F. A. Chalub, P. A. Markowich, and B. Perthame. {\it Kinetic models for chemotaxis and their drift-diffusion limits.} Springer Vienna, 2004.

\bibitem[F] {F}
W. E. Fitzgibbon. {\it Initial boundary value problems for the Carleman equation.} Computers \&\ Mathematics with Applications 9.3 (1983): 519-525.

\bibitem[F2] {F2}
W. E. Fitzgibbon. {\it The fluid dynamical limit of the Carleman equation with reflecting boundary.} Nonlinear Analysis: Theory, Methods \&\ Applications 6.7 (1982): 695-702.

\bibitem[PT] {PT}
A. Pulvirenti and G. Toscani. {\it Fast diffusion as a limit of a two-velocity kinetic model.} Rend. Circ. Mat. Palermo Suppl 45 (1996): 521-528.

\bibitem[SV]{SV}
F. Salvarani and J.L. Vazquez. {\it The diffusive limit for Carleman-type kinetic models.} Nonlinearity 18.3 (2005): 1223.

\bibitem[TL]{TL}
G. Toscani and P. L. Lions. {\it Diffusive limit for finite velocity Boltzmann kinetic models.} Revista Matematica Iberoamericana 13.3 (1997): 473-514.


\bibitem[GS] {GS}
F. Golse and S. Salvarani.{\it The nonlinear diffusion limit for generalized Carleman models: the initial-boundary value problem.} Nonlinearity 20.4 (2007): 927.

\bibitem [ST] {ST}
F. Salvarani and G. Toscani. {\it The diffusive limit of Carleman-type models in the range of very fast diffusion equations.} Journal of Evolution Equations 9.1 (2009): 67-80.


\bibitem[V] {V}
J. L. V\'{a}zquez. {\it The porous medium equation: mathematical theory.} Oxford University Press, 2007.
\bibitem[V2] {V2}
J. L. V\'{a}zquez. {\it Smoothing and decay estimates for nonlinear parabolic equations of porous medium type.} Oxford Lecture Series in Mathematics and its Applications, Oxford Univ. Press, Oxford (2006).

 \bibitem [DKS] {DKS}
P. Daskalopoulos, J. King, and N. Sesum. {\it Extinction profile of complete non-compact solutions to the Yamabe flow.} arXiv preprint arXiv:1306.0859 (2013).

\bibitem [DS] {DS}
P. Daskalopoulos and N. Sesum. {\it On the extinction profile of solutions to fast diffusion.} Journal f\"{u}r die reine und angewandte Mathematik (Crelles Journal) 2008.622 (2008): 95-119.

\bibitem [DS2] {DS2}
P. Daskalopoulos and N. Sesum. {\it Eternal solutions to the Ricci flow on $\R^2$.} International Mathematics Research Notices 2006 (2006): 83610.



\bibitem [DK] {DK}
P. Daskalopoulos and C. E. Kenig. {\it Degenerate diffusion-initial value problems and local
regularity theory.} Tracts in Mathematics 1, European Mathematical Society, 2007

\bibitem [BBDGV] {BBDGV}
A. Blanchet, M. Bonforte, J. Dolbeault, G. Grillo and J.L. V\'{a}zquez{\it Asymptotics of the fast diffusion equation via entropy estimates.} Archive for Rational Mechanics and Analysis 191.2 (2009): 347-385.

\bibitem [BDGV] {BDGV}
M. Bonforte, J. Dolbeault, G. Grillo and J.L. V\'{a}zquez. {\it Sharp rates of decay of solutions to the nonlinear fast diffusion equation via functional inequalities.} Proceedings of the National Academy of Sciences 107.38 (2010): 16459-16464.

\bibitem [H] {H}
K. M. Hui. {\it Existence of solutions of the equation $u_t=\Delta \log u$.} Nonlinear Analysis: Theory, Methods \&\ Applications 37.7 (1999): 875-914.
\bibitem [H2] {H2}
K. M. Hui. {\it Asymptotic behaviour of solutions of the fast diffusion equation near its extinction time.} arXiv preprint arXiv:1407.2696 (2014).
\bibitem [H3] {H3}
K. M. Hui. {\it On some Dirichlet and Cauchy problems for a singular diffusion equation.} Differential and Integral Equations 15.7 (2002): 769-804.

\bibitem [HK] {HK}
K. M. Hui and S. Kim. {\it Extinction profile of the logarithmic diffusion equation.} Manuscripta Mathematica 143.3-4 (2014): 491-524.
\bibitem [HK2] {HK2}
K.M. Hui and S. Kim. {\it Large-time behaviour of the higher-dimensional logarithmic diffusion equation.} Proceedings of the Royal Society of Edinburgh: Section A Mathematics 143.04 (2013): 817-830.

\bibitem [L] {L}
O. A. Ladyzenskaja. {\it Linear and quasilinear equations of parabolic type.} Transl. Math. Monographs 23 (1968).
\end{thebibliography}
\end{document}